\documentclass[12pt]{article}
\usepackage{amsfonts, amsmath, amssymb, amsgen, amsthm, amscd,latexsym,mathrsfs}
\usepackage{color}
\usepackage[all]{xy}

\def\Diff{\mathop{\rm Diff}\nolimits}

\def\Id{\mathop{\rm Id}\nolimits}

\def\Ad{\mathop{\rm Ad}\nolimits}
\def\ad{\mathop{\rm ad}\nolimits}

\def\Tr{\mathop{\rm Tr}\nolimits}

\def\Hom{\mathop{\rm Hom}\nolimits}

\def\Cb{{\mathbb C}}

\def\Nb{{\mathbb N}}
\def\Rb{{\mathbb R}}

\def\Ac{{\cal A}}

\def\Fc{{\cal F}}

\def\Hc{{\cal H}}

\def\Jc{{\cal J}}
\def\Ic{{\cal I}}

\def\Uc{{\cal U}}
\def\Vc{{\cal V}}

\def\Kc{{\cal K}}

\def\Cc{{\cal C}}

\def\Dc{{\cal D}}

\def\a{\alpha}
\def\b{\beta}
\def\d{\delta}
\def\D{\Delta}
\def\g{\gamma}

\def\G{\Gamma}

\def\om{\omega}

\def\s{\sigma}

\def\t{\theta}

\def\ve{\varepsilon}
\def\vp{\varphi}

\def\0b{\bf 0}

\def\nb{\nabla}
\def\ot{\otimes}

\def\ra{\rightarrow}

\def\lt{\triangleleft}

\def\acl{\blacktriangleright\hspace{-4pt}\vartriangleleft }
\def\bcl{\vartriangleright\hspace{-4pt} \blacktriangleleft}

\def\btl{\blacktriangleleft}

\def\hd{\overset{\ra}{\partial}}
\def \vd{\uparrow\hspace{-4pt}\partial}
\def\hs{\overset{\ra}{\sigma}}
\def \vs{\uparrow\hspace{-4pt}\sigma}
\def\hta{\overset{\ra}{\tau}}
\def \vta{\uparrow\hspace{-4pt}\tau}

\def\p{\partial}

\def\0D{\Delta^{(0)}}
\def\1D{\Delta^{(1)}}

\def\wg{\wedge}

\def\td{\tilde}
\def\cop{{^{\rm cop}}}
\newcommand{\wbar}[1]{\overline{#1}}

\newcommand{\FD}{\mathfrak{D}}

\newcommand{\Fa}{\mathfrak{a}}
\newcommand{\Fg}{\mathfrak{g}}
\newcommand{\Fh}{\mathfrak{h}}

\newcommand{\Fd}{\mathfrak{d}}

\newcommand{\Fs}{\mathfrak{s}}

\newcommand{\Ft}{\mathfrak{t}}

\newtheorem{theorem}{Theorem}[section]
\newtheorem{remark}[theorem]{Remark}
\newtheorem{proposition}[theorem]{Proposition}
\newtheorem{lemma}[theorem]{Lemma}

\newtheorem{example}[theorem]{Example}
\newtheorem{definition}[theorem]{Definition}

\def\ni{\noindent}

\def\build#1_#2^#3{\mathrel{
\mathop{\kern 0pt#1}\limits_{#2}^{#3}}}
\newcommand{\ps}[1]{~\hspace{-4pt}_{^{(#1)}}}

\newcommand{\ns}[1]{~\hspace{-4pt}_{_{{<#1>}}}}

\newcommand{\nsb}[1]{~\hspace{-4pt}_{^{[#1]}}}

\newcommand{\snb}[1]{~\hspace{-4pt}^{^{{[#1]}}}}

\def\odots{\ot\cdots\ot}
\def\wdots{\wedge\dots\wedge}

\def\one{{\bf 1}}

 \newcommand{\ie}{{\it i.e.\/}\ }

\def\a{\alpha}
\def\b{\beta}
\def\d{\delta}

\def\g{\gamma}

\def\om{\omega}
\def\s{\sigma}
\def\t{\theta}
\def\ve{\varepsilon}

\def\vp{\varphi}

\def\D{\Delta}
\def\G{\Gamma}

\def\nb{\nabla}

\def\ot{\otimes}
\def\part{\partial}

\def\ra{\rightarrow}

\def\text{\hbox}

\def\nb{\nabla}
\def\ot{\otimes}

\def\ra{\rightarrow}

\def\Ad{\mathop{\rm Ad}\nolimits}

\def\Diff{\mathop{\rm Diff}\nolimits}

\def\Hom{\mathop{\rm Hom}\nolimits}

\def\Id{\mathop{\rm Id}\nolimits}

\def\build#1_#2^#3{\mathrel{
\mathop{\kern 0pt#1}\limits_{#2}^{#3}}}

\newcommand{\FBOX}[2]
{\begin{center} \fcolorbox{black}{white}{\parbox{#1 cm}{ #2}}
\end{center}}

\numberwithin{equation}{section}
\begin{document}

\title{\bf  Characteristic classes of foliations via SAYD-twisted cocycles}
\author{
\begin{tabular}{cc}
Bahram Rangipour \thanks{Department of Mathematics  and   Statistics,
     University of New Brunswick, Fredericton, NB, Canada      Email: bahram@unb.ca, \quad and \quad ssutlu@ unb.ca }
     \quad and  \quad  Serkan S\"utl\"u $~^\ast$
      \end{tabular}}

\maketitle

\abstract{\noindent   We have previously shown  that  the truncated
Weil algebra of any Lie algebra is a Hopf-cyclic type complex with
nontrivial coefficients. In this paper we apply this  result  to
transfer   the characteristic classes of  transversely orientable
foliations into the cyclic  cohomology of the groupoid action
algebra. Our result in codimension 1 matches with  the only existing
explicit  computation  done by    Connes-Moscovici. In codimension 2 case,  we carry out a constructive and  explicit computation, by which  we present  the transverse  fundamental class, the Godbillon-Vey class,  and the other four residual  classes as cyclic cocycles on the groupoid action algebra.     The main object in charge  in this new characteristic map
is a SAYD-twisted cyclic cocycle of the same degree as the
codimension. We construct such a  cocycle  by introducing an equivariant Hopf-cyclic cohomology and an equivariant cup product.}

\tableofcontents

\section{Introduction}\label{section-intro}

 Following Connes-Moscovici \cite{ConnMosc98},  let
$\Ac_\G:=C^{\infty}_c(F^+)\rtimes \G$. Here $F^+$ is the oriented
frame bundle over $\Rb^n$, and $\G$ is a  discrete subgroup of
$\Diff^+(\Rb^n)$,  the group of orientation preserving
 diffeomorphisms of $\Rb^n$.

\medskip

\ni For an arbitrary  $\G$, the cyclic cohomology of $\Ac_\G$ is not
known  \cite[Sect. III.2]{Conn-book}. However, the Gelfand-Fuks
cohomology of $\Fa_n$, the Lie algebra of formal vector fields on
$\Rb^n$, is finite dimensional and is embedded  in this cohomology
as a direct summand. In other words, there is a  map, even in  the
level of complexes, which is  a composite of two complicated maps:
\begin{equation}
\xymatrix{ H_{\rm GF}(\Fa_n,\Cb)\ar[rr]^{\Phi\circ\Vc}\ar[rd]^{ \Vc_{\rm van Est}}& & HP(\Ac_\G)\\
& H_\tau(F^+, \Cb)\ar[ru]^{\Phi_{\rm Connes}}.&
 }
\end{equation}
The first map  is  a van Est type  map \cite{ConnMosc98}, which
lands in the twisted cohomology  computed  by the Bott bicomplex
\cite[Prop. III.2.11]{Conn-book}, while the second map is due to
Connes \cite[Thm. III.2.14]{Conn-book}. The representatives of the
Gelfand-Fuks cohomology
 classes in $H(\Fa_n,\Cb)$ are known thanks to the  Vey basis of the
 cohomology of  the truncated Weil algebra \cite{Godb72}. However it
 is difficult to transfer them to the cyclic cohomology of $\Ac_\G$.
  The reader is referred to \cite{ConnMosc} for a complete account of the computation in codimension 1.

\medskip

\ni The Hopf-cyclic cohomology, invented  by Connes-Moscovici
 for computing a local index formula \cite{ConnMosc98}, made it
possible to have another characteristic map
\begin{equation}\label{CM-char}
\chi_\tau: HP(\Hc_n, \Cb_\d)\ra HP(\Ac_\G).
\end{equation}
Here $\Hc_n$ is the Connes-Moscovici Hopf algebra of codimension $n$  and $\Cb_\d$ is
 the canonical one dimensional  SAYD module over $\Hc_n$.  One of the  interesting features  of this
 characteristic map  is its   simple  presentation on the level of
 complexes,
\begin{equation}
\chi_\tau(h_1\odots h_n)(a^0, \ldots, a^n)=\tau\big(a^0
h_1(a^1)\cdots h_n(a^n)\big).
\end{equation}
Here $\tau$ is the canonical trace on $\Ac_\G$ defined by
\begin{equation}\label{aux-canonical-trace}
 \tau(fU^\ast_\psi) = \left\{\begin{array}{cc}
  \displaystyle \int_{F^+}f\varpi,  & \text{if}\quad \psi = \Id, \\
  &\\
  0, & \text{otherwise.}
  \end{array}\right.
\end{equation}
It is also proved  that $HP(\Hc_n, \Cb_\d)$ and $H_{\rm GF}(\Fa_n, \Cb)$ are
 canonically isomorphic, although once again this isomorphism is not
 easy to present \cite{ConnMosc98,MoscRang11}. In view of
 \eqref{CM-char}, the only obstacle to transfer the
  characteristic classes of transversely orientable foliations to the cyclic
 cohomology of $\Ac_\G$ is a basis of the representatives
 of the Hopf-cyclic cohomology classes of  $\Hc_n$.
  There is an intensive ongoing  study
\cite{MoscRang07,MoscRang09,MoscRang11}  to investigate the Hopf-cyclic cohomology of the geometric bicrossed product Hopf algebras
such as $\Hc_n$.   The main idea is
  to use the bicrossed product construction  to find the smallest
complex by which one can compute the cohomology classes. This
complex is found in \cite{MoscRang11} and
 is shown to be the codomain  of  the van Est isomorphism \cite{MoscRang11}. However,
  the return map from that complex to the Hopf-cyclic complex of the
 Hopf algebra  is still missing.

\medskip

\ni In this paper we develop a new characteristic map, whose source
is the Hopf-cyclic cohomology of $\Kc:=U(g\ell_n)$, the enveloping
algebra of the general linear Lie algebra $g\ell_n$.
 The Hopf algebra   $\Kc$ obviously  is not as sophisticated as
$\Hc_n$. Therefore to obtain and transfer the  same classes, by considering the  conservation of work, we would expect
 a characteristic map and a SAYD module more sophisticated
than $\chi_\tau$ and $\Cb_\d$
respectively.

\medskip

   \ni  In fact,  the first step of our mission  was taken in
\cite{RangSutl-II}, where the authors showed that the
      truncated Weil algebra is  a  Hopf-cyclic complex. As a
      result, the characteristic classes of transversely orientable
      foliations can be calculated from $HC(\Kc, V)$. Here
      $V:=S(g\ell_n^\ast)\nsb{2n}$, the algebra of $n$-truncated
      polynomials on $g\ell_n$,
       is a canonical and nontrivial SAYD module over $\Kc$.

\medskip

\ni The other important  piece of this new characteristic map is a SAYD twisted   cyclic $n$-cocycle
 $\vp\in C^n_\Kc(\Ac_\G, V)$.
 Next, we apply  the cup product in Hopf cyclic cohomology introduced in \cite{KhalRang05-II} by  Khalkhali and the  first named author.  In fact, since  $\vp$ is a cyclic cocycle, we use  the explicit
   formula derived in \cite{Rang08}
  to compute    the characteristic classes of foliations as
   cyclic cocycles in $HC(\Ac_\G)$ via
\begin{equation}\label{cup}
\chi_\vp: HC^{\bullet}(\Kc, V)\ra HC^{\bullet+n}(\Ac_\G),\quad
\chi_\vp(x)=x\cup \vp.
\end{equation}

\ni In order to test  our method  we first carry out the computation
for  codimension 1 and observe  that our result matches with the classes
obtained by  Connes-Moscovici in \cite{ConnMosc}.  The result of
\cite{MoscRang11} shows that  the amount of work in codimension 2 is
not comparable with that of  codimension 1.  However, we completely
determine the representatives of all classes in $HC(\Kc,V)$ in
addition to an explicit formula for $\vp\in HC^2_\Kc(\Ac_\G, V)$.
Then \eqref{cup} yields our desired cyclic cocycles in
 $HC(\Ac_\G)$.

\medskip

\ni Throughout the paper, all vector spaces and their tensor
 products are over $\Cb$ unless otherwise is specified. We  use the Sweedler's notation for comultiplication and coaction.
   We denote the comultiplication of a coalgebra $C$ by $\D: C\ra C\ot C$  and its action on $c\in C$ by $\D(c) = c\ps{1} \ot c\ps{2}$. The image of $v\in V$ under  a left coaction $\nabla: V\ra C\ot V$   is denoted by $\nabla(v) = v\ns{-1} \ot v\ns{0}$, summation   suppressed. By the coassociativity, we simply write
$\D(c\ps{1}) \ot c\ps{2} = c\ps{1} \ot \D(c\ps{2}) = c\ps{1} \ot c\ps{2} \ot c\ps{3}.$
Unless stated otherwise, a Lie algebra $\Fg$ is finite dimensional with a basis $\{X_i\,|\,1 \leq i \leq n\}$ and a dual basis $\{\t^i\,|\,1 \leq i \leq n\}$. In particular, for $\Fg = g\ell_n$ we use $\{Y_i^j\,|\,1 \leq i,j \leq n\}$ for a basis and $\{\t^i_j\,|\,1 \leq i,j \leq n\}$ for a dual basis. We denote the Weil algebra of $\Fg$ by $W(\Fg)$, and $W(\Fg)\nsb{2n}$ stands for the $n$-truncated Weil algebra of $\Fg$. We denote the Kronecker symbol by $\d^i_j$.  We also adopt the Einstein summation convention on the repeating indices unless otherwise is stated. Finally, for the sake of  simplicity  we  use
\begin{equation*}
B_{\s(1)}\odots B_{\s(q)} := \sum_{\s \in S_q} {\rm sgn}(\s) B_{\s(1)}\odots B_{\s(q)}
\end{equation*}
for any set of objects $\{B_1,\ldots, B_q\}$. Here $S_q$ is the group of all permutations on $q$ objects and ${\rm sgn}(\s)$ stands for the signature of $\s$.

\section{Preliminaries}\label{section-prelim}
In this section we bring all  material needed  for the sequel
sections.  The definition of  Hopf-cyclic cohomology and a brief
account of the Connes-Moscovici characteristic map is provided in
the first  subsection. The basics of the  cyclic cohomology of Lie
algebras are recalled in the other subsection.

\subsection{Hopf-cyclic cohomology with coefficients}
 %The coefficient spaces  for Hopf-cyclic cohomology are called
%stable anti-Yetter-Drinfeld modules.
\ni Let $H$ be a Hopf algebra equipped with a character $\d: H\ra
\Cb$, \ie an algebra map,
 and a group-like element $\s\in H$, \ie $\D(\s)=\s\ot \s$ and $\ve(\s)=1$.
 The pair $(\d,\s)$ is called a modular pair in involution (MPI for short) if
\begin{equation}
\d(\s)=1, \quad \text{and}\quad  S_\d^2=\Ad_\s,
\end{equation}
where $\Ad_\s(h)= \s h\s^{-1}$, for any $h \in H$, and  $S_\d$ is
defined by
\begin{equation}\label{aux-twisted-antipode}
S_\d(h)=\d(h\ps{1})S(h\ps{2}), \quad h \in H.
\end{equation}

\medskip
\ni A vector space
 $M$ is called a right-left
stable-anti-Yetter-Drinfeld module (SAYD for short) over $H$ if it
is a right
 $H$-module, a left $H$-comodule, and
\begin{equation}\label{aux-SAYD-condition}
\nabla(m\cdot h)= S(h\ps{3})m\ns{-1}h\ps{1}\ot m\ns{0}\cdot h\ps{2},\qquad  m\ns{0}\cdot m\ns{-1}=m,
\end{equation}
for any $v\in V$ and $h\in H$. \ni Using $\d$ and $\s$ one endows
$^\s\Cb_\d:=\Cb$, the field of complex numbers, with a  right module
and left comodule structures  over $H$. This way $^\s\Cb_\d$ is a
SAYD module over the Hopf algebra $H$ if and only if $(\d,\s)$ is an
MPI.

\medskip

\ni Now let $M$  be  a right-left SAYD module over  $H$ and $C$ an $H$-module coalgebra, that is, $\D(h(c))=h\ps{1}(c\ps{1}) \ot h\ps{2}(c\ps{2})$ for any $h \in H$ and $c \in C$. Then we
have the graded space
\begin{equation}\label{aux-stand-Hopf-cyclic}
C_H(C,M) := \bigoplus_{q\geq 0} C^q_H(C,M), \quad C^q_H(C,M):= M\ot_H C^{\ot q+1}
\end{equation}
with the coface operators
\begin{align}\label{aux-coface-opt}
&\Fd_i: C^q_H(C,M)\ra C^{q+1}_H(C,M), \quad 0\le i\le q+1\\\notag
&\Fd_i(m\ot_H c^0\odots c^q)= m\ot c^0\odots \D(c^i )\odots c^q, \\\notag
&\Fd_{q+1}(m\ot_H c^0\odots c^q)=\\\notag
& m\ns{0}\ot_H c^0\ps{2} \ot c^1\odots c^q\ot m\ns{-1}(c^0\ps{1}),
\end{align}
the codegeneracy operators
\begin{align}\label{aux-codegeneracy-opt}
&\Fs_j: C^q_H(C,M)\ra C^{q-1}_H(C,M), \quad 0\le j\le q-1 \\\notag
&\Fs_j (m\ot_H c^0\odots h^q)= m\ot_H c^0\odots \ve(c^{j+1})\odots
c^q,
\end{align}
and the cyclic operator
\begin{align}\label{aux-cyclic-opt}
&\Ft_q: C^q_H(C,M)\ra C^q_H(C,M),\\\notag
&\Ft_q(m\ot_H c^0\odots c^q)=m\ns{0}\ot_H c^1\odots c^q\ot m\ns{-1}(c^0).
\end{align}
The graded space $C_H(C,M)$ endowed with the above operators forms
a cocyclic module.
% which means that $\p_i,$ $\s_j$ and $\tau$ satisfy
%\begin{equation}
%\d_j  \d_i = \d_i  \d_{j-1}, \, \, i < j  , \qquad \s_j \s_i = \s_i
%\s_{j+1},  \, \,  i \leq j
%\end{equation}
%\begin{equation}
%\s_j  \d_i = \left\{ \begin{matrix} \d_i  \s_{j-1} \hfill &i < j
%\hfill \cr 1_n \hfill &\hbox{if} \ i=j \ \hbox{or} \ i = j+1 \cr
%\d_{i-1}  \s_j \hfill &i > j+1  ;  \hfill \cr
%\end{matrix} \right.
%\end{equation}
%\begin{eqnarray}
%\tau_n  \d_i  = \d_{i-1}  \tau_{n-1} ,
% \quad && 1 \leq i \leq n ,  \quad \tau_n  \d_0 =
%\d_n \\ \label{cj} \tau_n  \s_i = \s_{i-1} \tau_{n+1} , \quad &&
%1 \leq i \leq n , \quad \tau_n  \s_0 = \s_n  \tau_{n+1}^2 \\
%\label{ce} \tau_n^{n+1} &=& 1_n  \, .
%\end{eqnarray}
\ni Using the above operators one defines the Hochschild coboundary
$b$ and  the Connes boundary operator $B$,
\begin{align}
&b: C^{q}_H(C,M)\ra C^{q+1}_H(C,M), \qquad
b:=\sum_{i=0}^{q+1}(-1)^i\Fd_i, \\
&B: C^{q}_H(C,M)\ra C^{q-1}_H(C,M), \quad
B:=\left(\sum_{i=0}^{q}(-1)^{qi}\Ft^{i}\right) \Fs_{q-1}\Ft.
\end{align}
The cyclic cohomology of $C$ under the symmetry of $H$
with coefficients in the SAYD module $M$, which is  denoted by
$HC(C,M)$, is defined to be the cyclic cohomology of the complex $C_H(C,M)$.

\medskip

\ni For $C = H$, the map
\begin{align}
\begin{split}
&\Jc: M\ot_H H^{\ot (n+1)}\ra M\ot H^{\ot n},\\
&\Jc(m\ot_H h^0\odots h^n)=mh^0\ps{1}\ot S(h\ps{2})\cdot(h^1\odots
h^n).
\end{split}
\end{align}
identifies the standard Hopf-cyclic complex \eqref{aux-stand-Hopf-cyclic} with
\begin{equation}
C(H,M) := \bigoplus_{q\geq 0} C^q(H,M), \quad C^q(H,M):= M\ot H^{\ot q}.
\end{equation}
Then the coface operators become
\begin{align*}
&\Fd_i: C^q(H,M)\ra C^{q+1}(H,M), \quad 0\le i\le q+1\\
&\Fd_0(m\ot h^1\odots h^q)=m\ot 1\ot h^1\odots h^q,\\\notag
&\Fd_i(m\ot h^1\odots h^q)=\\
& m\ot h^1\odots h^i\ps{1}\ot h^i\ps{2}\odots h^q, \quad 1\le i\le q \\
&\Fd_{q+1}(m\ot h^1\odots h^q)=m\ns{0}\ot h^1\odots h^q\ot m\ns{-1},
\end{align*}
the codegeneracy operators
\begin{align*}
&\Fs_j: C^q(H,M)\ra C^{q-1}(H,M), \quad 0\le j\le q-1 \\
&\Fs_j (m\ot h^1\odots h^q)= m\ot h^1\odots \ve(h^{j+1})\odots
h^q,
\end{align*}
and the cyclic operator becomes
\begin{align*}
&\Ft: C^q(H,M)\ra C^q(H,M),\\
&\Ft(m\ot h^1\odots h^q)=m\ns{0}h^1\ps{1}\ot
S(h^1\ps{2})\cdot(h^2\odots h^q\ot m\ns{-1}).
\end{align*}

\medskip

\ni Let $A$ be a  $H$-module algebra, that is,    a (left) $H$-module
and
\[h(ab)=h\ps{1} (a) h\ps{2}(b),\quad h(1_A)=\ve(h)1_A, \quad \forall h\in H,  a\in A.\]
   One endows $V\ot A^{\ot n+1}$ with the action of $H$
\begin{equation}
(v\ot a^0\odots q^q)\cdot h= mh\ps{1}\ot S(h\ps{q+2})a^0\odots S(h\ps{2})a^q,
\end{equation}
and  forms
\begin{equation}\label{aux-HC(A,V)}
C^n_H(A,V)=\Hom_H(V\ot A^{\ot n+1}, \Cb)
\end{equation}
as the space of  $H$-linear maps.  It is checked in  \cite{HajaKhalRangSomm04-II}
 that for any $v \ot \td{a}:=v \ot a^0 \odots a^{n+2} \in V \ot A^{\ot\,n+2}$ the  morphisms
\begin{align*}
&(\p_i \vp)(v\ot \td{a}) = \vp(v\ot a^0\ot \dots
\ot a^i a^{i+1}\ot\dots \ot a^{n+1}), \quad 0 \leq i <n,\\
&  (\p_{n+1}\vp)(v\ot \td{a}) =\vp(v\ns{0}\ot
(S^{-1}(v\ns{-1})a^{n+1})a^0\ot a^1\ot\dots \ot a^{n}),\\
&(\s_i\vp)(v\ot \td a) = \vp(v\ot a^0 \ot \dots
\ot a^i\ot 1 \ot \dots \ot a^{n-1}), \quad 0\le i\le n-1,\\
&(\tau\vp)(v\ot \td a) = \vp(v\ns{0}\ot (S^{-1}(v\ns{-1})a^n)\ot
a^0\ot \dots \ot a^{n-1})
\end{align*}
define a cocyclic module structure on $C^n_H(A,V)$, whose cyclic
cohomology  is denoted by $HC_H(A,V)$.

\medskip

\ni One uses $HC(H,V)$ and $HC_H(A,V)$ to define a cup product
\begin{align*}
&HC^p_H(A,V)\ot HC^q_H(H,V)\ra HC^{p+q}(A),
\end{align*}
whose definition can be found in \cite{Rang08,KhalRang05-II}.

\medskip
%\ni  A linear map $\tau:A \to \Cb$
%is called $\d$-invariant
%$\s$-trace if
%\begin{equation}
%\tau\big(h(a)\big)=\d(h)\tau(a), \quad\qquad \tau(ab)=\tau(b\s(a)).
%\end{equation}
%If $A$ is equipped with a $\d$-invariant $\s$-trace, then
\ni As the simplest example,   one notes that the cup product with the 0-cocycle $\tau \in C^0_H(A,\;
^\s\Cb_\d)$ defines the Connes-Moscovici characteristic map
\cite{ConnMosc98,ConnMosc},
\begin{align}\label{aux-Connes-Moscovici-charac-map}
\begin{split}
& \chi_\tau:HC^\bullet(H, ^\s\Cb_\d) \to HC^\bullet(A)\\
& \chi_\tau(h^1 \odots h^n)(a^0 \odots a^n)=\tau(a^0h^1(a^1)\ldots
h^n(a^n)).
\end{split}
\end{align}

\subsection{ Lie algebra (co)homology}
In this subsection, after recalling the  Lie algebra (co)homology,
we summarize  our work in \cite{RangSutl-II} on the cyclic
cohomology of Lie algebras with coefficients in SAYD modules.

\medskip

\ni Let  $\Fg$ be a Lie algebra and $V$ be a right  $\Fg$-module.
The Lie algebra homology complex is defined to be
\begin{equation}
C(\Fg, V) = \bigoplus_{q\geq0}C_q(\Fg, V), \quad C_q(\Fg, V) := \wg^q \Fg \ot V
\end{equation}
with the Chevalley-Eilenberg boundary map
\begin{align}
&\xymatrix{\cdots \ar[r]^{\p_{\rm CE}\;\;\;\;\;\;\;\;} & C_2(\Fg,V)
\ar[r]^{\p_{\rm CE}\;\;\;} & C_1(\Fg,V)
\ar[r]^{\;\;\;\;\;\;\;\p_{\rm CE}} & V},\\\notag
 & \p_{\rm CE}(X_0 \wdots X_{q-1} \ot v ) = \sum_{i= 0}^{q-1} (-1)^i
X_0 \wdots \widehat{X}_i \wdots X_{q-1} \ot v \cdot X_i + \\\notag
 & \sum_{0 \leq i<j \leq q-1} (-1)^{i+j} [X_i, X_j] \wg X_0 \wdots
\widehat{X}_i \wdots \widehat{X}_j \wdots X_{q-1} \ot v.
\end{align}
The homology of the complex $(C(\Fg, V),\p_{\rm CE})$ is called the
Lie algebra homology of $\Fg$ with coefficients in $V$ and it is
denoted by $H_{\bullet}(\Fg,V)$.  In a dual fashion one defines  the
Lie algebra cohomology complex
\begin{equation}
W(\Fg, V) = \bigoplus_{q\geq0}W^q(\Fg, V),
 \quad W^q(\Fg,V)=\Hom(\wedge^q \Fg ,V),
\end{equation}
where $\Hom(\wedge^q \Fg ,V)$ is the vector space
 of all alternating linear maps on $\Fg^{\ot q}$
 with values in $V$. The Chevalley-Eilenberg coboundary
\begin{equation}
\xymatrix{V\ar[r]^{d_{\rm CE}\;\;\;\;\;\;\;\;}
&W^1(\Fg,V)\ar[r]^{d_{\rm CE}}& W^2(\Fg,V)
\ar[r]^{\;\;\;\;\;\;d_{\rm CE}}&\cdots\;,}
\end{equation}
is defined by
 \begin{align}\label{aux-Chevalley-Eilenberg-coboundary}
 \begin{split}
& d_{\rm CE}(\a)(X_0, \ldots,X_q)=\sum_{0 \leq i<j \leq q}
(-1)^{i+j}\a([X_i,X_j], X_0\ldots \widehat{X}_i, \ldots,
 \widehat{X}_j, \ldots, X_q)+\\
& ~~~~~~~~~~~~~~~~~~~~~~~\sum_{i=0}^q (-1)^{i+1}\a(X_0,\ldots,
\widehat{X}_i,\ldots X_q)\cdot X_i.
\end{split}
 \end{align}
Alternatively, we may identify  $W^q(\Fg,V)$ with $\wg^q\Fg^\ast \ot V$
and the coboundary $d_{\rm CE}$ with
\begin{align}\notag
& d_{\rm CE}(v)= - \t^i \ot v \cdot X_i,\quad  d_{\rm CE}(\b \ot v)=
d_{\rm dR}(\b)\ot v  - \t^i \wg \b \ot v \cdot X_i, \\
& d_{\rm dR}:\wedge^p\Fg^\ast\ra \wedge^{p+1}\Fg^\ast, \quad d_{\rm
dR}(\t^i)=-\frac{1}{2}C^i_{jk}\t^j\wg\t^k
\end{align}
The cohomology of the complex $(W(\Fg,V),d_{\rm CE})$, the Lie
algebra cohomology of $\Fg$ with coefficients in $V,$  is denoted by
$H^\bullet(\Fg,V).$

\medskip
\ni In this paper  we are particularly interested in the SAYD
modules over the universal enveloping algebra $U(\Fg)$ of a Lie
algebra $\Fg$. By our results in \cite{RangSutl-II},  such SAYD
modules are in fact obtained from the SAYD modules over the Lie
algebra $\Fg.$
\begin{definition}[\cite{RangSutl-II}]
A vector space $V$ is a left comodule over the Lie algebra $\Fg$
 if there is a linear map
\begin{equation}
\nabla_{\Fg}: V \ra \Fg \ot V,\qquad \nabla_{\Fg}(v)=v\nsb{-1}\ot v\nsb{0}
\end{equation}
such that
\begin{equation*}
v\nsb{-2}\wg v\nsb{-1}\ot v\nsb{0}=0,
\end{equation*}
where
\begin{equation*}
v\nsb{-2}\ot v\nsb{-1}\ot v\nsb{0}= v\nsb{-1}\ot (v\nsb{0})\nsb{-1}\ot
 (v\nsb{0})\nsb{0}.
\end{equation*}
\end{definition}
\ni It is clear that  left $\Fg$-comodules and  right
$S(\Fg^*)$-modules are identical.
\begin{definition}[\cite{RangSutl-II}]
Let $V$ be a right module and left comodule over a Lie algebra
$\Fg$. We call $V$ a right-left anti-Yetter-Drinfeld module (AYD
module) over $\Fg$ if
\begin{equation}
\nabla_{\Fg}(v \cdot X) = v\nsb{-1} \ot v\nsb{0} \cdot X +
[v\nsb{-1}, X] \ot v\nsb{0}.
\end{equation}
Moreover, $V$ is called stable if
\begin{equation}
v\nsb{0} \cdot v\nsb{-1} = 0.
\end{equation}
Finally, $V$ is said to be unimodular stable if $V_{-\d}$ is stable.
Here the character $\d$ is defined by  $\d = \Tr \circ \ad:\Fg \to
\Cb$ and  $V_{-\d}$ is the  deformation of $V$ via
\[ v\lt X:= v\cdot X-\d(X)v.\]

\end{definition}

\begin{example}\label{example-symm-alg-SAYD}{\rm
 The truncated polynomial algebra $V = S(\Fg^\ast)\nsb{2n}$, of a Lie
algebra $\Fg$, is a unimodular SAYD module over $\Fg$ with the
coadjoint action and the Koszul coaction defined by
\begin{equation}\label{aux-Koszul-ex}
\nabla_{K}: V \to \Fg \ot V, \quad \nabla_{K}( v)= \sum_{i=1}^nX_i \ot
v\t^i.
\end{equation}
}\end{example}

\ni Via the help of (unimodular) SAYD modules  we generalize  Lie
algebra (co)homology complexes. Let us start  with the  Lie algebra
homology  by  introducing  the  complex
\begin{equation}\label{aux-complex-C(g,V)}
C(\Fg,V) = \bigoplus_{i \geq 0}\wg^i\Fg\ot V, \qquad \p = \p_{\rm
CE} + \p_{\rm K}
\end{equation}
with the Chevalley-Eilenberg boundary and the Koszul coboundary
\begin{equation}
\p_{\rm K}:C_n(\Fg,V) \to C_{n+1}(\Fg,V), \qquad \p_{\rm K}(e \ot
v)= v\nsb{-1} \wg e \ot v\nsb{0}.
\end{equation}
%It is shown in \cite{RangSutl-II} that $V$ induces a SAYD module
%over $U(\Fg)$ and that $H(C(\Fg,V)) = HP(U(\Fg),V)$.

\ni  Applying the Poincar\'e duality  in the Lie algebra
cohomology  to the complex \eqref{aux-complex-C(g,V)}, see
\cite[Prop. 4.4]{RangSutl-II}, we obtain
\begin{equation}
W(\Fg,V) = \bigoplus_{i \geq 0}\wg^i\Fg^\ast\ot V
\end{equation}
with $d = d_{\rm CE} + d_{\rm K}$, where $d_{\rm CE}:W^n(\Fg,V) \to W^{n+1}(\Fg,V)$ is the Chevalley-Eilenberg coboundary and
\begin{equation*}
d_{\rm K}:W^n(\Fg,V) \to W^{n-1}(\Fg,V), \qquad d_{\rm K}(\alpha \ot
v)= \iota(v\nsb{-1})(\alpha) \ot v\nsb{0}.
\end{equation*}
Here $\iota(X)$ is the contraction with respect to $X$.

\medskip

 \ni In particular, we recover  the (truncated) Weil algebra
\cite{RangSutl-II}:
\begin{equation}
W(\Fg,S(\Fg^\ast)) = W(\Fg), \qquad W(\Fg,S(\Fg^\ast)\nsb{2n}) = W(\Fg)\nsb{2n}.
\end{equation}
%Finally, for the Lie algebra $\Fg = gl(n)$ and the (unimodular) SAYD module $V = S(gl(n)^\ast)\nsb{2n}$, we have
%\begin{equation}
%\xymatrix{
%& HP(U(gl(n)),S(gl(n)^\ast)\nsb{2n}) \ar[r]^\cong & H(C(gl(n),S(gl(n)^\ast)\nsb{2n})) \ar[d]^\cong \\
%HP(\Hc_n)  & \ar[l]_\cong H_{\rm GF}(\Fa_n)  & \ar[l]_\cong H(W(gl(n))\nsb{2n})\,. &
%}
%\end{equation}
%The vertical isomorphism is the Poincar\'e isomorphism of Lie
% algebra homology-cohomology, and the lower horizontal
 %isomorphisms are the canonical isomorphism of
 %Connes-Moscovici \cite{ConnMosc98} and Gelfand-Fuks \cite{GelfFuks70}.

%hhh

\section{SAYD-twisted cyclic cocycles}
In this section we fix $K$ to be a cocommutative Hopf subalgebra of
a Hopf algebra $H$. Let $A$ be an $H$-module algebra, and  $V$ be a
SAYD module over $K$. We aim to develop a
machinery to produce SAYD-twisted cyclic cocycles in $HC_K(A,V)$. In
the first subsection we introduce equivariant Hopf-cyclic cohomology
$HC_K(H,V,N)$, where $N$ is a SAYD module over $H$. In the second
subsection we construct a cup product $HC^p_K( H, V,N)\ot HC^q_
H(A,N)\ra HC^{p+q}_K(A,V)$. In the third and fourth subsection we
apply the results of the first two subsections. This way we produce
 a nontrivial SAYD-twisted cyclic cocycle  over the groupoid action
algebra under the symmetry of the general linear Lie algebra with
coefficients in the truncated polynomials on this Lie algebra.

\subsection{Equivariant Hopf-cyclic cohomology}
\label{section-equiv-cohom}   For a SAYD module $N$ over $H$ and a
right module-left comodule  $V$ over $K$ we define  the graded space
\begin{align}\label{aux-equivariant-complex}
\begin{split}
& C_K(H, V,N)=\bigoplus_{q \geq 0}C^q_K(H, V,N), \\
& \Cc^q:=C^q_K(H, V,N):=\Hom_K(V, N\ot_H  H^{\ot q+1}).
\end{split}
\end{align}
More precisely,  $\phi\in \Cc^q$ if and only  if  for any $u\in K$
and  any $v\in V$
\begin{equation}\label{aux-equivariancy}
\phi(v\cdot u)= \phi(v)\cdot u,
\end{equation}
where the right  action of $K$ on $N\ot_H H^{\ot q+1}$ is the usual
diagonal action,  \ie
\begin{equation}
(n\ot_H h^0\odots h^{q})\cdot u= n\ot_H h^0 u\ps{1}\odots
h^{q}u\ps{q+1}.
\end{equation}
For $\phi\in C^q_K(H, V,N)$ and $v\in V$,  we use the notation
\begin{equation}\label{aux-notation-phi(m)}
\phi(v)=\phi(v)\snb{-1}\ot_H \phi(v)\snb{0} \odots \phi(v)\snb{q}.
\end{equation}
Let us define the morphisms  $\Fd_i:\Cc^q\ra \Cc^{q +1}$,
$\Fs_j:\Cc^{q}\ra\Cc^{q-1} $, and $\Ft_q:\Cc^q\ra \Cc^q$ as
\begin{align}
\begin{split}
&d_i(\phi)(v)= \Fd_i(\phi(v)), \quad 0\le i\le q\\[.5cm]
&d_{q+1}(\phi)(v)= \Fd_{q+1}(\phi(v\ns{0}))\lt S(v\ns{-1}),\\[.5cm]
&s_j(\phi)(v)=\Fs_j(\phi(v)), \quad 0\le j\le q-1,\\[.5cm]
&t_q(\phi)(v)=\Ft_q(\phi(v\ns{0}))\lt S(v\ns{-1}),
\end{split}
\end{align}
where the right action $\lt$ of $K$ on $N\ot_H H^{\ot q+1}$ is defined
by
\begin{equation}
(n\ot_H h^0\odots h^q) \lt u= n\ot_H h^0\odots h^{q-1}\ot h^q u.
\end{equation}
Here  the morphisms  $(\Fd_i,\Fs_j, \Ft)$ are the usual morphisms of
the cocyclic module  $C_H(H, N)$ defined in
 \eqref{aux-coface-opt},\eqref{aux-codegeneracy-opt} and \eqref{aux-cyclic-opt}.

\begin{theorem}
If $V$  and $N$ are SAYD modules over $K$ and $H$ respectively, then
 the morphisms $d_i,s_j$ and $t$ define  a cocyclic module structure
on $C_K(H, V,N)$.
\end{theorem}

\begin{proof}
Let us prove that the morphisms $d_i, s_j$, and  $t$  are
well-defined. Indeed, it suffices to check
  that $t$, $d_0$, and $s_n$ are well-defined as the other morphisms are made
   of these three.  For $d_0 $ and $s_n$ the task is obvious as
$\D:H\ra H\ot H$ and $\ve:H\ra \Cb$ are multiplicative respectively.
Let us check that $t$ is well-defined. We have
 \begin{align}
 \begin{split}
 &t(\phi)(v\cdot y)=\tau(\phi((v\cdot y) \ns{0}))\lt S((v\cdot y)\ns{-1})\\
 &= \tau(\phi(v\ns{0}\cdot y\ps{2}) )\lt S(( S(y\ps{1})v\ns{-1}y\ps{3})\\
 &= \tau(\phi(v\ns{0})\cdot y\ps{2} )\lt S(y\ps{3}) S(v\ns{-1})y\ps{1}\\
  &=\tau(\phi(v)\snb{-1}\ot_H \phi(v)\snb{0} y\ps{2} \odots \phi(v)\snb{q} y\ps{q+2} )\lt S(y\ps{q+3}) S(v\ns{-1})y\ps{1}\\
  &= \phi(v)\snb{-1}\ns{0}\ot_H \phi(v)\snb{1} y\ps{3} \ot\dots\\
  &\dots \ot \phi(v)\snb{q} y\ps{q+2}\ot \phi(v)\snb{-1}\ns{-1} \phi(v)\snb{0} y\ps{2} S(y\ps{q+3}) S(v\ns{-1})y\ps{1}\\
  &= \phi(v)\snb{-1}\ns{0}\ot_H \phi(v)\snb{1} y\ps{1} \ot\dots\\
  &\dots\ot \phi(v)\snb{q} y\ps{q}\ot \phi(v)\snb{-1}\ns{-1} \phi(v)\snb{0}  S(v\ns{-1})y\ps{q+1}\\
  &=t(\phi(v))\cdot y.
\end{split}
 \end{align}
\ni In the second and the sixth equalities we use the fact that  $K$
is cocommutative.

\medskip

\ni Let us next show that $C_K(H,V,N)$ is a cocyclic module which means that $d_i,$ $s_j$ and $t$ satisfy
\begin{equation}\label{aux-disj-order}
d_j  d_i = d_i  d_{j-1}, \, \, i < j  , \qquad s_j s_i = s_i
s_{j+1},  \, \,  i \leq j
\end{equation}
\begin{equation}\label{aux-disj-order-II}
s_j  d_i = \left\{ \begin{matrix} d_i  s_{j-1} \hfill &i < j
\hfill \cr \Id_q \hfill &\hbox{if} \ i=j \ \hbox{or} \ i = j+1 \cr
d_{i-1}  s_j \hfill &i > j+1  ;  \hfill \cr
\end{matrix} \right.
\end{equation}

\begin{eqnarray}\label{aux-t-d-order}
t_{q+1}  d_i  = d_{i-1}  t_q ,
 \quad && 1 \leq i \leq q+1 ,  \quad t_{q+1}  d_0 =
d_{q+1} \\ \label{cj} t_{q-1}  s_i = s_{i-1} t_q , \quad &&
1 \leq i \leq q-1 , \quad t_q  s_0 = s_{q-1}  t_q^2 \\
\label{ce} t_q^{q+1} &=& \Id_q  \, .
\end{eqnarray}
The equalities \eqref{aux-disj-order}, \eqref{aux-disj-order-II} and \eqref{cj} follow directly from their counterparts for the operators $\Fd_i, \Fs_j$ and $\Ft$.

\medskip

\ni As for \eqref{aux-t-d-order}, we check the case $i = q+1$.
Indeed
\begin{align}
\begin{split}
& t_{q+1} (d_{q+1}(\phi))(v) = \Ft_{q+1}(d_{q+1}(\phi)(v\ns{0}))\lt S((v\ns{-1}))\\
&= \Ft_{q+1}(\Fd_{q+1}(\phi(v\ns{0}\ns{0}))\lt S(v\ns{0}\ns{-1}))\lt S(v\ns{-1})\\
&= \Ft_{q+1}(\Fd_{q+1}(\phi(v\ns{0}))\lt S(v\ns{-1}))\lt S(v\ns{-2})\\
& = \Ft_{q+1}\left(\phi(v\ns{0})\snb{-1}\ns{0}\ot_H \phi(v\ns{0})\snb{0}\ps{2} \ot \phi(v\ns{0})\snb{1} \odots\right. \\
& \left. \phi(v\ns{0})\snb{q} \ot  \phi(v\ns{0})\snb{-1}\ns{-1}\phi(v\ns{0})\snb{0}\ps{1}S(v\ns{-1})\right)\lt S(v\ns{-2})\\
&= \phi(v\ns{0})\snb{-1}\ns{0}\ot_H \phi(v\ns{0})\snb{1} \odots \phi(v\ns{0})\snb{q} \ot \\
& \phi(v\ns{0})\snb{-1}\ns{-2}\phi(v\ns{0})\snb{0}\ps{1}S(v\ns{-1}) \ot \\
& \phi(v\ns{0})\snb{-1}\ns{-1}\phi(v\ns{0})\snb{0}\ps{2}S(v\ns{-2}) \\
&= \phi(v\ns{0})\snb{-1}\ns{0}\ot_H \phi(v\ns{0})\snb{1} \odots \phi(v\ns{0})\snb{q} \ot \\
& \D\left(\phi(v\ns{0})\snb{-1}\ns{-1}\phi(v\ns{0})\snb{0}S(v\ns{-1})\right) \\
& = \Fd_q\left(\left[\phi(v\ns{0})\snb{-1}\ns{0}\ot_H \phi(v\ns{0})\snb{1} \odots \phi(v\ns{0})\snb{q} \ot \right.\right.\\
& \left.\left.\phi(v\ns{0})\snb{-1}\ns{-1}\phi(v\ns{0})\snb{0}\right]\lt S(v\ns{-1})\right)\\
& = \Fd_q(\Ft_q(\phi(v\ns{0}))\lt S(v\ns{-1})) = d_q(t_q(\phi))(v).
\end{split}
\end{align}
A simple calculation yields one to
\begin{align*}
& t_q^{q+1}(\phi)(v)\\
&= \phi(v\ns{0})\snb{-1}\ns{0}\ot_H \phi(v\ns{0})\snb{-1}\ns{-q-1}
\phi(v\ns{0})\snb{0}S(v\ns{-1})\ot\dots\\
& \dots\ot
\phi(v\ns{0})\snb{-1}\ns{-1}\phi(v\ns{0})\snb{q}S(v\ns{-q-1})\ot\dots\\
&=\phi(v\ns{0})\snb{-1}\ns{0} \phi(v\ns{0})\snb{-1}\ns{-1}\ot_H \phi(v\ns{0})\snb{0}S(v\ns{-1})\\
& \dots\ot \phi(v\ns{0})\snb{q}S(v\ns{-q-1})\\
&=\phi(v\ns{0})\snb{-1}\ot_H \phi(v\ns{0})\snb{0}S(v\ns{-1})
 \odots \phi(v\ns{0})\snb{q}S(v\ns{-q-1})\\
&=\phi(v\ns{0})\cdot S(v\ns{-1})= \phi(v\ns{0}\cdot S(v\ns{-1}))=\phi(v)\\
\end{align*}
The last equality is held because of the fact that for any SAYD
module $V$ and any $v \in V$, $v\ns{0}\cdot S^{-1}(v\ns{-1})=v$, see \cite[Lemma
4.9]{KhalRang03}. Here Since $K$ is cocommutative $S=S^{-1}$.
 \end{proof}

\ni The cyclic cohomology of $C_K(H,V,N)$ is denoted by
$HC_K(H,V,N)$.

\medskip

\ni One notes that by taking $K=\Cb$ and $V=\Cb$ one recovers the
usual Hopf-cyclic cohomology $HC(H,N).$

\subsection{Equivariant characteristic map}
\label{subsection-eq-char} Let $V$ and $N$ be  SAYD modules over $K$
and $ H$ respectively. We define the map
  \begin{align}
 &\Psi: C^q_K( H,V,N)\ot C^q_ H(A,N)\longrightarrow  C^q_K(A,
V),\\\notag
 &\Psi(\phi\ot\psi)(v\ot a_0\odots a_q)\\
 &= \psi\big( \phi(v)\snb{-1}\ot  \phi(v)\snb{0}(a_0)\ot \phi(v)\snb{1}( a_1) \odots \phi(v)\snb{q}( a_q)\big).
  \end{align}
One may check that $\Psi$ is a map of cocyclic modules,  where  on the left hand side we consider the product of two cocyclic modules. This is enough  to produce  a generalization of the  cup product in Hopf-cyclic cohomology \cite{KhalRang05-II,Rang08}.

\medskip

\ni We define a new bicocyclic module by
tensoring the  cocyclic modules \eqref{aux-cyclic-opt}, and \eqref{aux-HC(A,V)}. The new bigraded  module
 in the bidegree  $(p,q)$ is defined by
\begin{equation}\label{c**}\Cc^{p,q}:=\Hom_K(V, N\ot_H H^{\ot p+1})\ot \Hom_H(A^{\ot q+1},N)\end{equation} with horizontal structure
$\hd_i=\Fd_i\ot \Id$, $\hs_j=\Fs\ot \Id$, and $\hta=\Ft\ot \Id$ and
vertical structure $\vd_i= \Id\ot\p_i$, $\vs_j=\Id\ot \s_j$, and
$\vta=\Id\ot \tau$. Obviously $(\Cc^{\bullet,\bullet},\hd,\hs,\hta,\vd,\vs,\vta)$ defines a bicocyclic module.

\medskip

\ni Now let us define  the  map
\begin{align}\label{acpsi}
&\Psi:\Dc^q\ra \Hom_K(V\ot A^{\ot q+1},\Cb),\\\notag
&\Psi(\phi\ot\psi)(v\ot a_0\odots a_q)\\\notag
 &= \psi\big( \phi(v)\snb{-1}\ot  \phi(v)\snb{0}(a_0)\ot \phi(v)\snb{1}( a_1) \odots \phi(v)\snb{q}( a_q)\big).
\end{align}
Here $\Dc^{\bullet}$ denotes the diagonal complex of the
bicocyclic module  $\Cc^{\bullet,\bullet}$. It is a cocyclic
module whose $q$th component is $\Cc^{q,q}$ and its cocyclic structure
morphisms are  $\p_i:=\hd_i\circ\vd_i$, $\s_j:=\hs_j\circ\vs_j$, and
$\tau:=\hta\circ\vta$.

\begin{proposition} \label{cyclic-map}
The map $\Psi$ is  a well-defined map of cocyclic modules.
\end{proposition}

\begin{proof}
Let us  first show that $\Psi$ is well-defined. Indeed, by using the
fact that  $\phi$ is  $K$-linear,  we see
\begin{align*}
&\Psi(\phi\ot \psi)\big(vk\ps{1} \ot  S(k\ps{q+2})(a^0)\odots S(k\ps{2})(a^{q})\big)\\
&=\psi( \phi(vk\ps{1})\snb{-1}\ot  \phi(vk\ps{1})\snb{0}S(k\ps{q+2})(a^0)\odots      \phi(vk\ps{1})\snb{q} S(k\ps{2})(a^{q}))\\
&= \psi( \phi(v)\snb{-1}\ot  \phi(v)\snb{0}k\ps{1}S(k\ps{2q+2})(a^0)\odots      \phi(v)\snb{q}k\ps{q+1} S(k\ps{q+2})(a^{q}))\\
&= \ve(k)\psi( \phi(v)\snb{-1}\ot  \phi(v)\snb{0}(a^0)\odots      \phi(v)\snb{q}(a^{q}))\\
&=\ve(k)\Psi(\phi\ot \psi)\big(v \ot  a^0\odots a^{q}\big).
\end{align*}
 Next, we show that $\Psi$  commutes with the  cocyclic structure
 morphisms. To this end,  we need only to show the commutativity of
 $\Psi$ with zeroth coface, the last codegeneracy and the cyclic
 operator, because these  operators generate all cocyclic structure
 morphisms. We check it only  for the cyclic operators and leave the
 rest to the reader.
\begin{align*}
&\tau \Big(\Psi(\phi\ot \psi)\Big)\big(v \ot a^0\odots a^{q}\big)\\
&= \Psi(\phi\ot \psi)\big(v\ns{0} \ot  S^{-1}(v\ns{-1})(a^q)\ot a^0\odots a^{q-1}\big)\\
&=\psi(\phi(v\ns{0})\snb{-1}\ot \phi(v\ns{0})\snb{0}S^{-1}(v\ns{-1})(a^q)\ot \phi(v\ns{0})\snb{1}(a^1)\ot\dots\\
&\dots\ot \phi(v\ns{0})\snb{q}(a^q).
\end{align*}
On the other hand we have
\begin{align*}
&\Psi(\Ft\phi\ot \tau\psi)\big(v \ot a^0\odots a^{q}\big)\\
&= \tau\psi( \Ft\phi(v)\snb{-1}\ot  \Ft\phi(v)\snb{0}(a^0)\odots \Ft\phi(v)\snb{q}(a^q))\\
&=\tau\psi( \Ft\phi(v)\snb{-1}\ot  \Ft\phi(v)\snb{0}(a^0)\odots \Ft\phi(v)\snb{q}(a^q))\\
&=\tau\psi( \phi(v\ns{0})\snb{-1}\ns{0}\ot  \phi(v)\snb{1}(a^0)\ot\dots\\
 &~~~~~~~\dots\ot \phi(v)\snb{q}(a^{q-1})\ot \phi(v\ns{0})\snb{-1}\ns{-1}\phi(v\ns{0})\snb{0} S(v\ns{-1})(a^q))\\
 & =\psi\big( \phi(v\ns{0})\snb{-1}\ns{0}\ot  S^{-1}(\phi(v\ns{0})\snb{-1}\ns{-1})\phi(v\ns{0})\snb{-1}\ns{-2}\phi(v\ns{0})\snb{0} S(v\ns{-1})(a^q)\ot \\
 &  \phi(v)\snb{1}(a^0)\odots \phi(v)\snb{q}(a^{q-1})\big)\\
 &=\psi(\phi(v\ns{0})\snb{-1}\ot \phi(v\ns{0})\snb{0}
 S(v\ns{-1})(a^q)\ot \phi(v\ns{0})\snb{1}(a^1)\odots \phi(v\ns{0})\snb{q}(a^q).
\end{align*}
Since $K$ is cocommutative $S^2=\Id$, and hence $$\tau\big(\Psi(\phi\ot \psi)\big)=\Psi(\Ft\phi\ot \tau\psi). $$
\end{proof}

\begin{theorem}
Assume that  $K$ is a cocommutative Hopf subalgebra of  a Hopf
algebra $ H$,   $\Ac$  is a $ H$-module algebra,  and  $V$ and $N$
are SAYD modules over $K$ and $ H$ respectively.  Then  the map
$\Psi$ defines a cup product
\begin{equation}\label{aux-equ-cup}
HC^p_K( H, V,N)\ot HC^q_ H(A,N)\ra HC^{p+q}_K(A,V).
\end{equation}
\end{theorem}
\begin{proof}
Let $[\phi]\in HC^p_K(H,V,N)$ and $[\psi]\in HC^q_H(A,N)$. Without
loss of generality we assume that  $\phi$ and $\psi$  are
respectively cyclic cocycles horizontally and  vertically. This
implies
 that $\phi\ot \psi$ is a $(b,B)$ cocycle of degree $p+q$ in total
  compex of  $\Cc^{\bullet,\bullet}$. On the other hand by the cyclic
  Eilenberg-Zilber theorem \cite{KhalRang04}, the total
  complex of $\Cc^{\bullet,\bullet}$ is quasi-isomorphic
   with $\Dc^{\bullet}$  via the Alaxander-Witney map $AW$.
     So, $AW(\phi\ot \psi)$ is a $(b, B)$ cocycle
     in $\Dc^{\bullet}$. Since $\Psi$ is cyclic, we conclude that
     $\Psi(AW(\phi\ot \psi))$ defines a
 class in $HC^{p+q}_K(A, V)$.
\end{proof}
\ni One notes that by setting $K:=\Cb$ as the trivial  Hopf subalgebra
and $M=\Cb$ as the trivial SAYD module over $K$  the above cup
product becomes  the cup product defined in
\cite{KhalRang05-II,Rang08}.

\subsection{Equivariant charactrestic map for
$\Hc_n$}\label{subsection-H-n} In this subsection we apply our
equivariant characteristic map we built  in  Subsection
\ref{subsection-eq-char} to produce the desired  cyclic cocycle on
the groupoid action algebra.

\medskip

 \ni Let us first  recall the Connes-Moscovici Hopf algebra
$\Hc_n$ from \cite{ConnMosc98,ConnMosc}. To this end let $\Fh_n$ be
the Lie algebra generated by
\begin{equation} \label{gens}
\{X_k , \, Y_i^j , \, \d_{jk \,  \ell_1  \ldots \ell_r}^i \, ;  \,
i, j, k, \ell_1  \ldots \ell_r  =1, \ldots , n , \, r \in \Nb\}
\end{equation}
with relations
\begin{align}\label{aux-relations-CM}
\begin{split}
& [Y_i^j , Y_k^{\ell}] = \d_k^j Y_i^{\ell} - \d_i^{\ell} Y_k^j, \quad  [Y_i^j , X_k]=
\d_k^j X_i, \quad [X_k ,
X_{\ell}] = 0,\\
& \d_{jk   \ell_1 \ldots \ell_r}^i = [X_{\ell_r} , \ldots
[X_{\ell_1} , \d_{jk}^i] \ldots ],\quad [\d_{jk  \ell_1 \ldots \ell_r}^i , \, \d_{j'k'   {\ell}'_1
\ldots  {\ell}'
_r}^{i'}]= 0,\\
&  [Y_p^q , \d_{j_1 j_2 \,  j_3 \ldots j_r}^i ]  = \sum_{s=1}^r \, \d^q_{j_s} \,
 \d^i_{j_1  j_2 \,  j_3  \ldots j_{s-1} p j_{s+1} \ldots j_r}
  - \d_p^i \, \d_{j_1 j_2 \,  j_3  \ldots j_r}^q, \\
  & \d_{ jk\ell_1 \,   \ldots \ell_r}^i=  \d_{jk\ell_{\pi(1)} \ldots
  \ell_{\pi(r)}}^i, \quad \forall \pi\in S_r.
\end{split}
\end{align}
As an algebra, $\Hc_n$ is $U(\Fh_n)$ modulo the (Bianchi-type) identities
\begin{equation}
 \d_{j \ell   k}^i -  \d_{j k   \ell}^i =  \d_{j k}^s \, \d_{s \ell}^i  -  \d_{j \ell}^s \, \d_{s k}^i .
\end{equation}
The coalgebra structure of $\Hc_n$ is defined by a Leibniz  rule
that makes the  groupoid action algebra $\Ac=
C^{\infty}_c(F^+)\rtimes \G$ an $\Hc_n$-module algebra.

\medskip

\ni In order to describe the action of $\Hc_n$ explicitly, let us
first identify $F^+$ with $\Rb^n\rtimes {\rm GL}^+(n,\Rb)$ and use
the local coordinates $(x,y) \in F^+$. A typical element of the
groupoid action algebra is of the form $fU^\ast_\phi$, where
$U^\ast_\phi$ stands for $\phi^{-1} \in \G$ and $f\in
C^\infty_c(F^+)$. The elements of $\Hc_n$ act as the following
operators.
\begin{align}\label{aux-action-Hn-on-A}
\begin{split}
& X_k = y_k^\mu\frac{\p}{\p x^\mu}, \qquad X_k(fU^\ast_\phi) := X_k(f)U^\ast_\phi,\\
& Y_i^j = y_i^\mu\frac{\p}{\p y_\mu^j}, \qquad Y_i^j(fU^\ast_\phi) := Y_i^j(f)U^\ast_\phi,\\
& \d_{jk \,  \ell_1 \ldots \ell_r}^i \, ( f \, U_{\phi}^\ast) = \g_{jk \,  \ell_1 \ldots \ell_r}^i (\phi)\, f \, U_{\phi}^\ast,
\end{split}
\end{align}
where
\begin{align}
\begin{split}
& \g_{jk \,  \ell_1 \ldots \ell_r}^i (\phi) = X_{\ell_r} \cdots X_{\ell_1} \big(\g_{jk}^i (\phi)\big),\\
& \g_{jk}^i (\phi) (x, y)  = \left( y^{-1} \cdot
{\phi}^{\prime} (x)^{-1} \cdot \part_{\mu} {\phi}^{\prime} (x) \cdot
y\right)^i_j \, y^{\mu}_k.
\end{split}
\end{align}
Therefore, for any $a,b\in \Ac$ we have the Leibniz rule
\begin{align}
\begin{split}
& Y_i^j (a b)  = Y_i^j (a) \, b \, + \, a \,Y_i^j (b), \\
& X_k(a  b)  =  X_k (a) \, b \, + \, a \,X_k (b)  + \d_{jk}^i (a) \, Y_{i}^{j} (b),\\
& \d_{jk}^i (a  b) =  \d_{jk}^i  (a) \, b  +  a \, \d_{jk}^i  (b).
\end{split}
\end{align}
Accordingly, we have
\begin{align}\label{aux-D-Y}
& \D(Y_i^j) = Y_i^j \ot 1+1 \ot Y_i^j,\\\label{aux-D-d}
& \D(\d_{jk }^i) = \d_{jk}^i \ot 1 + 1 \ot \d_{jk}^i,\\\label{aux-D-X}
& \D(X_k) = X_k \ot 1 + 1 \ot X_k + \d^i_{jk} \ot Y_i^j.
\end{align}
For simplicity, we will also employ the notation
\begin{equation}
{\d^a}_{k  \ell_1 \ldots \ell_r} :=\d_{jk  \ell_1 \ldots \ell_r}^i, \quad Y_a :=Y_i^j
\qquad a=\left(\begin{matrix}i\\j\end{matrix}\right).
\end{equation}

\medskip

\ni Throughout this subsection we  set  $ \Hc:= \Hc_n$,  the
Connes-Moscovici Hopf algebra, and $\Kc:=U(\Fg_0)$, where
$\Fg_0:=g\ell_n$.  We also let $N=\Cb_\d$ the  SAYD
module over $ \Hc$ where $\d: \Hc\ra \Cb$ is the character defined on
the generators by
\begin{equation}
\d(Y_i^j)=\d_i^j, \qquad \d(X_k)=\d(\d_{jk   \ell_1 \ldots \ell_r}^i)=0, \qquad 1\le i,j,k, l_t\le n
\end{equation}
and is extended  on $ \Hc$ multiplicatively. Finally we set $V=
S(\Fg_0^\ast)_{[2n]}$ with the canonical $\Kc$-SAYD module structure
as recalled in \eqref{example-symm-alg-SAYD}.

\medskip

\ni Let $\Ac:=\Ac_\G= C^{\infty}_c(F^+)\rtimes \G$ and $\tau$ be the
canonical trace  on $\Ac$ defined in \eqref{aux-canonical-trace}.
Since $\tau\in C^0_\Hc(\Ac, \Cb_\d)$ and $\tau$ is a cyclic cocycle
\cite{ConnMosc98}, applying the cup product \eqref{aux-equ-cup} we
get the  map of cocyclic modules
\begin{align}\label{aux-equivariant-map}
\begin{split}
&\chi_\tau^{\rm eq}:HC^q_\Kc( \Hc,V,\Cb_\d)\longrightarrow HC^q_\Kc(\Ac, V)\\
&\chi^{\rm eq}_\tau(\phi)(v\ot a_0\odots a_q)= \tau\left(\phi(v)
\snb{0}(a_0)\cdots \phi(v)\snb{q}(a_q)\right).
\end{split}
\end{align}
We conclude this section by the identification of
$C^q_\Kc(\Kc,V,\Cb_\d)$ with $(V^\ast\ot  \Cb_\d\ot \Hc^{\ot\, q})^{\Fg_0}$,
where $V^\ast=\Hom_\Cb(V, \Cb)$ and  $\Fg_0$ acts on $V^\ast\ot
\Hc^{\ot\, q}$ via
\begin{align}
\begin{split}
&(\phi\ot \one\ot  h^1\odots h^q)Z=- \sum^q_{i=1}
 \phi\ot \one\ot h^1\odots \Ad_Z(h^i)\odots h^q\\
&+\phi\ot \d(Z)\ot h^1\odots h^q + \phi\cdot Z\ot
 \one\ot  h^1\odots h^q.
\end{split}
\end{align}
Here, the action of $\Fg_0$ on $V^\ast$ is defined by $(\phi\cdot
Z)(v)=-\phi(v\cdot Z)$.

\medskip

\ni The aforementioned identification is defined  by the map
\begin{align}\label{iso}
\begin{split}
&\Ic: (V^\ast\ot \Cb_\d\ot  \Hc^{\ot\, q})^{\Fg_0}\ra C_\Kc^q( \Hc, V,\Cb_\d)\\
&\Ic(\phi\ot \one\ot h^1\odots h^q)(v)= \phi(v)\ot_\Hc 1_ \Hc\ot
h^1\odots h^q.
\end{split}
\end{align}

\begin{proposition}
The map $\Ic$ defined in \eqref{iso} is an isomorphism of vector spaces.
\end{proposition}

\begin{proof}
Let us first check that $\Ic$ is well-defined. Indeed,
\begin{align*}
&\Ic(\phi\ot \one \ot h^1\odots h^q)(v\cdot Z)=\phi(v\cdot Z)\ot_ \Hc 1_ \Hc\ot h^1\odots h^q \\
&= -(\phi\cdot Z)(v)\ot_ \Hc 1_ \Hc\ot h^1\odots h^q\\
&=  -\sum^q_{i=1} \phi(v)\ot_ \Hc 1_ \Hc\ot h^1\odots \Ad_Z(h^i)\odots h^q\\
&+\d(Z)\phi(v)\ot_ \Hc 1_ \Hc\ot h^1\odots h^q\\
&=  -\sum^q_{i=1} \phi(v)\ot_ \Hc 1_ \Hc\ot h^1\odots \Ad_Z(h^i)\odots h^q\\
&+   \phi(v)\ot_ \Hc Z\ot h^1\odots h^q  +\sum_{i=1}^q \phi(v)\ot_ \Hc 1_ \Hc\ot h^1\odots Zh^i\odots  h^q\\
&= (\phi(v)\ot 1_ \Hc\ot h^1\odots h^q)\cdot Z=  (\Ic(\phi\ot \one \ot
h^1\odots h^q)(v))\cdot Z.
\end{align*}
Next,  we introduce an inverse map  for $\Ic$. To this end we fix a
basis for $V$, say $\{v_1, \ldots, v_m\}$, with a dual basis
$\{\nu^1, \ldots, \nu^m\}$ for $V^\ast$, and we define
\begin{align}
&\Ic^{-1}: C_\Kc^q( \Hc, V,\Cb_\d)\ra  (V^\ast\ot\Cb_\d\ot   \Hc^{\ot
q})^{\Fg_0}\\\notag
 &\Ic^{-1}(\phi)=\\\notag
 &\sum_{i=1}^m \nu^i \ot
\phi(v_i)\snb{-1}\d(\phi(v_i)\snb{0}\ps{1})\ot
S(\phi(v_i)\snb{0}\ps{2})\cdot \Big(\phi(v_i)\snb{1})\odots
\phi(v_i)\snb{q}\Big).
\end{align}
It is straightforward to check that $\Ic^{-1}$ is independent of the
choice of bases and is inverse to $\Ic$.
\end{proof}

\subsection{A SAYD-twisted cyclic  cocycle in codimension
1}\label{section-codim-1} In this subsection we  keep the setting of
Subsection \ref{subsection-H-n} for $n=1$.  Our aim is to introduce
an equivariant cyclic 1-cocycle $\vp\in C^1_\Kc(\Ac,V)$. In order to
apply \eqref{aux-equivariant-map},   we shall consider the complex
$C^1_\Kc(\Hc,V,\Cb_\d)$, which in turn is identified with the
$\Fg_0$-invariant subspace $(V^\ast \ot \Cb_\d \ot \Hc)^{\Fg_0}$ via
\eqref{iso}.

\medskip

\ni Let $\{R\}$ be the  basis for $\Fg_0^\ast$ as the dual basis of
$\{Y:=Y^1_1\}$ for $\Fg_0$. Let also $\{1, R\}$ be the basis of
$V$ and  $\{1^\ast,S\}$ as  the  dual basis for $V^\ast$.

\medskip

We define $ \vp_0,\vp_1:V \ot \Ac^{\ot\;2} \to \Cb,$ by
\begin{equation}
 \vp_0 = \chi_\tau^{\rm eq}(1^\ast \ot \one \ot X), \quad \vp_1 =
\chi_\tau^{\rm eq}(S \ot \one \ot \d_1),
\end{equation}
 more precisely
\begin{align}\label{aux-pieces-of-cocycle}
\begin{split}
&\vp_0((\a1 + \b\t) \ot a_0 \ot a_1) = \a\tau(a_0X(a_1)), \\
& \vp_1((\a1 + \b\t) \ot a_0 \ot a_1) = \b\tau(a_0\d_1(a_1)).
\end{split}
\end{align}
\begin{lemma}
The linear maps  $\vp_0,\vp_1:V \ot \Ac^{\ot\;2} \to \Cb$ defined in
\eqref{aux-pieces-of-cocycle} are $\Kc$-equivariant.
\end{lemma}

\begin{proof}
We have
\begin{align}
\begin{split}
& \vp_0((\a1 + \b\t) \ot Y(a_0) \ot a_1) + \vp_0((\a1 + \b\t) \ot a_0 \ot Y(a_1)) = \\
& \a\tau(Y(a_0)X(a_1)) + \a\tau(a_0X(Y(a_1))) = \\
& \a\tau(Y(a_0)X(a_1)) + \a\tau(a_0Y(X(a_1))) - \a\tau(a_0X(a_1)) = \\
& \a\tau(Y(a_0)X(a_1)) - \a\tau(Y(a_0)X(a_1)) = 0,
\end{split}
\end{align}
 the third equality follows from the integration by parts property \cite{ConnMosc98}
 \begin{equation}\label{aux-property-S-tilde}
\tau(h(a)b) = \tau(aS_\d(h)(b)), \quad \forall h\in \Hc,
\forall a,b\in \Ac.
\end{equation}
Similarly we have
\begin{align}
\begin{split}
& \vp_1((\a1 + \b\t) \ot Y(a_0) \ot a_1) + \vp_1((\a1 + \b\t) \ot a_0 \ot Y(a_1)) \\
& =\a\tau(Y(a_0)\d_1(a_1)) + \a\tau(a_0\d_1(Y(a_1)))  \\
& =\a\tau(Y(a_0)\d_1(a_1)) + \a\tau(a_0Y(\d_1(a_1))) - \a\tau(a_0\d_1(a_1))  \\
& =\a\tau(Y(a_0)\d_1(a_1)) - \a\tau(Y(a_0)\d_1(a_1)) = 0.
\end{split}
\end{align}

\end{proof}

\begin{lemma}
The linear map  $\vp_0 - \vp_1:V \ot \Ac^{\ot\;2} \to \Cb$ is
a Hochschild 1-cocycle.
\end{lemma}
\begin{proof}
We have
\begin{align}
\begin{split}
& b(\vp_0 - \vp_1)((\a1 + \b\t) \ot a_0 \ot a_1 \ot a_2) = (\vp_0 - \vp_1)((\a1 + \b\t) \ot a_0a_1 \ot a_2) \\
& - (\vp_0 - \vp_1)((\a1 + \b\t) \ot a_0 \ot a_1a_2) + (\vp_0 - \vp_1)((\a1 + \b\t) \ot a_2a_0 \ot a_1) \\
& - (\vp_0 - \vp_1)(\a\t \ot Y(a_2)a_0 \ot a_1) \\
& = \a\tau(a_0a_1X(a_2)) - \a\tau(a_0X(a_1a_2)) + \a\tau(a_2a_0X(a_1)) + \a\tau(Y(a_2)a_0\d_1(a_1)) \\
& - \b\tau(a_0a_1\d_1(a_2)) + \b\tau(a_0\d_1(a_1a_2)) -
\b\tau(a_2a_0\d_1(a_1))=0.
\end{split}
\end{align}
Here we have used \eqref{aux-D-d} and \eqref{aux-D-X}.
\end{proof}

\begin{proposition}
The Hochschild cocycle  $\vp_0 - \vp_1:V \ot \Ac^{\ot\;2} \to \Cb$,
defined in \eqref{aux-pieces-of-cocycle}, is cyclic.
\end{proposition}
\begin{proof}
By using the $\d$ invariancy of $\tau$, \eqref{aux-D-d} and \eqref{aux-D-X}
we have
\begin{align*}
& t(\vp_0 - \vp_1)((\a1 + \b\t) \ot a_0 \ot a_1) = (\vp_0 - \vp_1)((\a1 + \b\t) \ot a_1 \ot a_0) \\
& - (\vp_0 - \vp_1)(\a\t \ot Y(a_1) \ot a_0) \\
& = \a\tau(a_1X(a_0)) + \a\tau(Y(a_1)\d_1(a_0)) - \b\tau(a_1\d_1(a_0)) \\
& = - \a\tau(a_0X(a_1)) + \b\tau(a_0\d_1(a_1))=-(\vp_0 -
\vp_1)((\a1 + \b\t) \ot a_0 \ot a_1).
\end{align*}
\end{proof}

\subsection{A SAYD-twisted cyclic  cocycle in codimension
2}\label{section-codim-2} Similar to the previous subsection, we
keep the setting of Subsection \ref{subsection-H-n} for $n=2$. Our
goal is to find a nontrivial  cyclic 2-cocycle  $\phi\in
C^2_\Kc(\Ac, V)$.

\medskip

\ni Similar to the case $n=1$,  we apply
\eqref{aux-equivariant-map}
 by considering $C^2_\Kc(\Hc,V,\Cb_\d)
\cong (V^\ast \ot \Cb_\d \ot \Hc^{\ot\,2})^{\Fg_0}$.

\medskip

\ni Let $\{R^i_j\;\mid 1\le i,j\le 2\}$ be the dual basis of
$g\ell_2$ with the pairing $\langle
Y^j_i,R^k_l\rangle= \d^j_k\d^i_l$. We take
\begin{equation}
\left\{1, R^i_j, R^k_lR^p_q\,\left|\;\;\left(\begin{array}{c}
                                      k \\
                                      l
                                    \end{array}
\right.\right) \leq \left(\begin{array}{c}
                                      p \\
                                      q
                                    \end{array}
\right)\right\},
\end{equation}
as  a basis for  $V$ which is simplified by $ \left\{1, R^a,
R^{ab}\,|\,a\leq b\right\}$. The  dual basis for $V^\ast$ is
expressed by $ \left\{1^\ast, S_a, S_{ab}\,|\,a\leq b\right\}.$

\medskip

\ni We recall from \cite{RangSutl-II} that the Koszul coaction \eqref{aux-Koszul-ex} gives rise to a $\Kc$-coaction by the formula
\begin{align}\label{aux-Koszul-K-coaction}
\begin{split}
& \nabla_{K}: V \to \Kc \ot V,\\
& \nabla_{K}(1)=1 \ot 1 + Y_a \ot R^a + \frac{1}{2!}Y_aY_b \ot R^{ab},\\
& \nabla_{K}(R^a)=1 \ot R^a + Y_b \ot R^{ab},\\
& \nabla_{K}(R^{ab})=1 \ot R^{ab}.
\end{split}
\end{align}

\medskip

\ni We decompose $V=V_0\oplus V_1\oplus V_2 $, where $V_0=\Cb\langle
1\rangle$, $V_1=\Cb\langle R^a \rangle$, and $V_2=\Cb\langle R^{ab}
\rangle$. Using this decomposition, any $\psi\in \Hom(V\ot \Ac^{q+1},
\Cb)$ is decomposed uniquely  as $\psi=\psi_0+\psi_1+\psi_2$ by
$\psi_i= \psi|_{V_i\ot \Ac^{\ot q+1}}$.

\medskip

 \ni We now consider the linear map  $\psi: V\ot \Ac^{\ot 3}\ra \Cb$  with
components
\begin{align}   \label{aux-vp-0}
& \psi_0 :=
\chi_\tau^{\rm eq}\left(\g_1 1^\ast \ot \one \ot X_{\s(1)} \ot X_{\s(2)} +
\g_2 1^\ast \ot \one \ot {\d^a}_{\s(1)} \ot X_{\s(2)}Y_a\right. \\\notag &
+\left.\g_3 1^\ast \ot \one \ot {\d^a}_{\s(1)}{\d^b}_{\s(2)}Y_b \ot Y_a +\g_4 1^\ast
\ot \one \ot {\d^a}_{\s(1)\s(2)} \ot Y_a\right).\\ \label{aux-vp-2}
& \psi_2 := \chi_\tau^{\rm eq}\left(\a_1S_{ab} \ot \one \ot {\d^a}_{\s(1)}
\ot {\d^b}_{\s(2)}\right. + \\\notag & \hspace{4cm} \left.\a_2S_{ab}
\ot \one \ot {\d^b}_{\s(1)} \ot {\d^a}_{\s(2)}\right),
\end{align}
\begin{align}\notag
& \psi_1:=
\chi_\tau^{\rm eq}\left(\b_1S_a \ot \one \ot {\d^a}_{\s(1)} \ot X_{\s(2)}
+ \b_2S_a\ot \one\ot X_{\s(1)} \ot {\d^a}_{\s(2)}\right. \\\notag & +
\b_3S_a \ot  \one \ot {\d^a}_{\s(1)}{\d^b}_{\s(2)} \ot Y_b + \b_4S_a \ot
\one\ot Y_b \ot {\d^a}_{\s(1)}{\d^b}_{\s(2)}\\\label{aux-vp-1} & + \b_5 S_a
\ot  \one \ot {\d^a}_{\s(1)}Y_b \ot {\d^b}_{\s(2)} + \b_6S_a \ot \one \ot
{\d^b}_{\s(1)}Y_b \ot {\d^a}_{\s(2)} \\\notag &  + \b_7 S_a\ot \one \ot
{\d^a}_{\s(1)} \ot {\d^b}_{\s(2)}Y_b + \b_8S_a \ot \one \ot {\d^b}_{\s(1)}
\ot {\d^a}_{\s(2)}Y_b \\\notag & + \left.\b_9S_a \ot \one\ot
\D({\d^a}_{\s(1)\s(2)})\right).
\end{align}
Our aim is to  determine the coefficients $\a_i,\b_j,\g_k$, such
that  $\psi$ is a cyclic 2-cocycle. To do so we prove a series of
technical lemmas.

\begin{lemma}\label{lemma-psi-eq}
The cochain $\psi$ is $\Kc$-equivariant.

\end{lemma}

\begin{proof}
We first simply observe that
\begin{align}\label{aux-gl_2-invariance-1}
 &\d^j_{\s(1)}B_i \ot B_{\s(2)}
  + \d^j_{\s(2)}B_{\s(1)} \ot B_i =
   \d_i^j \left(B_{\s(1)} \ot
B_{\s(2)}\right).
\end{align}

\ni Using  \eqref{aux-relations-CM}, in view of the action of
$g\ell_2$ on the Hopf algebra $\Hc$,
  the equivariancy condition follows from
\begin{align*}
 & \ad(Y_i^j)\left(\d^p_{qr} \ot
Y^q_p\right) = \d^j_r\left(\d^p_{qi} \ot Y^q_p\right), \\\notag &
\ad(Y_i^j)\left(\d^k_{tr} \ot \d^t_{ln}\right) =
\\ & \d^j_l\left(\d^k_{tr} \ot
\d^t_{in}\right) - \d^k_i\left(\d^j_{tr} \ot \d^t_{ln}\right) +
\d^j_r\left(\d^k_{ti} \ot \d^t_{ln}\right) + \d^j_n\left(\d^k_{tr}
\ot \d^t_{li}\right).
\end{align*}

\end{proof}

\ni We first observe that $\psi \in C^2_\Kc(\Ac,V)$ is a
 Hochschild cocycle on $V_2 \ot \Ac^{\ot\,n}$.

\begin{lemma}\label{lemma-bvp2=0}
For any $\a_i,\b_j,\g_k$, $(b\psi)_2=0$.
\end{lemma}

\begin{proof}
The result follows directly from the application of the Hochschild
coboundary map and the fact that ${\d^a}_{k}$ are derivations of
$\Ac$.
\end{proof}
\ni On the next move, we determine $\a_i$, $1 \leq i \leq 2$, in such a way that $\psi \in C^2_\Kc(\Ac,V)$ is a cyclic cocycle on $V_2 \ot \Ac^{\ot\,n}$.

\begin{lemma}\label{lemma-tauvp2=vp2}
For $\a_1 = \a_2$, we have $(\tau\psi)_2 = \psi_2$.
\end{lemma}

\begin{proof}
By definition of the cyclic operator, we have
\begin{equation}\label{aux-vp2-cyclic}
\tau\psi(R^{ab} \ot a_0\ot a_1 \ot a_2) = \psi_2(R^{ab} \ot a_2 \ot a_0\ot a_1).
\end{equation}
Hence, by the integration by parts property \eqref{aux-property-S-tilde},
\begin{align}
\begin{split}
& t\psi(R^{ab} \ot a_0 \ot a_1\ot a_2) =\\
& \a_1 \tau\left(a_0(-{\d^a}_{\s(1)}){\d^b}_{\s(2)}(a_1)a_2\right) + \a_1 \tau\left( a_0{\d^b}_{\s(2)}(a_1)(-{\d^a}_{\s(1)})(a_2)\right) \\
& +  \a_2\tau\left( a_0 (-{\d^b}_{\s(1)}){\d^a}_{\s(2)}(a_1) a_2\right) + \a_2 \tau\left(a_0 {\d^a}_{\s(2)}(a_1) (-{\d^b}_{\s(1)})(a_2)\right) \\
& = \a_1 \tau\left(a_0{\d^b}_{\s(1)}(a_1)  {\d^a}_{\s(2)}(a_2)\right) + \a_2\tau\left(a_0  {\d^a}_{\s(1)}(a_1) {\d^b}_{\s(2)}(a_2)\right).
\end{split}
\end{align}
Therefore, $(\tau\psi)_2 = \psi_2$ if and only if $\a_1 = \a_2$.
\end{proof}
\ni  As a result  we set
 \begin{equation}\label{aux-alpha}
 \a_1 = \a_2 = r.
\end{equation}
\medskip

\ni On the next step, we find a constraint on $\b_j$'s
 such that $\psi$ is a Hochschild cocycle over $V_1 \ot \Ac^{\ot\,q+1}$.

\begin{lemma}\label{lemma-bvp1=0}
We have $(b\psi)_1=0$ if and only if $\{\b_j\,|\,1 \leq j \leq 9\}$ satisfy the system
\begin{align}\label{system-vp1-Hochschild}
\begin{split}
 \b_1-\b_3+\b_7+r &= 0 \\
 \b_3+\b_8+r &= 0 \\
 -\b_2-\b_6+\b_8 &= 0 \\
 \b_4-\b_5 &= 0 \\
 -\b_4-\b_6& = 0 \\
 -\b_5+\b_7 &= 0.
\end{split}
\end{align}
\end{lemma}

\begin{proof}
Let us remind the reader that we have to use the Koszul coaction  \eqref{aux-Koszul-K-coaction} in the last coface
operator.
\begin{align}\label{aux-vp1-Hochschild}
\begin{split}
& b\psi(R^a \ot a_0\ot a_1 \ot a_2 \ot a_3) = \\
& \psi_1(R^a \ot a_0a_1 \ot a_2 \ot a_3) - \psi_1(R^a \ot a_0\ot a_1a_2 \ot a_3) \\
& + \psi_1(R^a \ot a_0\ot a_1 \ot a_2a_3) - \psi_1(R^a \ot a_3a_0\ot a_1 \ot a_2) \\
& + \psi_2(R^{ab} \ot Y_b(a_3)a_0\ot a_1 \ot a_2).
\end{split}
\end{align}
Therefore, as a result of the tracial property \cite[Thm. 6]{ConnMosc}
and the faithfullness \cite[(3.12)]{ConnMosc} of the trace, we have $(b\psi)_1 = 0$ if and only if
\begin{align}
\begin{split}
& \b_1 1\ot {\d^a}_{\s(1)} \ot {\d^b}_{\s(2)} \ot Y_b - \b_2 1\ot {\d^b}_{\s(1)} \ot Y_b \ot {\d^a}_{\s(2)} \\
& -\b_3 \left(1\ot {\d^a}_{\s(1)} \ot {\d^b}_{\s(2)} \ot Y_b + 1\ot {\d^b}_{\s(2)} \ot {\d^a}_{\s(1)} \ot Y_b\right) \\
& +\b_4 \left(1\ot Y_b \ot {\d^a}_{\s(1)} \ot {\d^b}_{\s(2)} + 1\ot Y_b \ot {\d^b}_{\s(2)} \ot {\d^a}_{\s(1)}\right) \\
& -\b_5 \left(1\ot {\d^a}_{\s(1)} \ot Y_b \ot {\d^b}_{\s(2)} + 1\ot Y_b \ot {\d^a}_{\s(1)} \ot {\d^b}_{\s(2)}\right) \\
& -\b_6\left(1\ot {\d^b}_{\s(1)} \ot Y_b \ot {\d^a}_{\s(2)} + 1\ot Y_b \ot {\d^b}_{\s(1)} \ot {\d^a}_{\s(2)}\right) \\
& +\b_7\left(1\ot {\d^a}_{\s(1)} \ot {\d^b}_{\s(2)} \ot Y_b + 1\ot {\d^a}_{\s(1)} \ot Y_b \ot {\d^b}_{\s(2)}\right) \\
& +\b_8\left(1\ot {\d^b}_{\s(1)} \ot {\d^a}_{\s(2)} \ot Y_b + 1\ot {\d^b}_{\s(1)} \ot Y_b \ot {\d^a}_{\s(2)}\right) \\
& + r 1 \ot {\d^a}_{\s(1)} \ot {\d^b}_{\s(2)} \ot Y_b + r 1 \ot {\d^b}_{\s(1)} \ot {\d^a}_{\s(2)} \ot Y_b = 0.
\end{split}
\end{align}
In other words,
\begin{align*}
& \left(\b_1-\b_3+\b_7+k\right) 1\ot {\d^a}_{\s(1)} \ot {\d^b}_{\s(2)} \ot Y_b + (\b_3+\b_8+k) 1 \ot {\d^b}_{\s(1)} \ot {\d^a}_{\s(2)} \ot Y_b \\
& + (-\b_2-\b_6+\b_8) 1\ot {\d^b}_{\s(1)} \ot Y_b \ot {\d^a}_{\s(2)} + (\b_4-\b_5) 1\ot Y_b \ot {\d^a}_{\s(1)} \ot {\d^b}_{\s(2)}\\
& + (-\b_4-\b_6) 1\ot Y_b \ot {\d^b}_{\s(1)} \ot {\d^a}_{\s(2)} + (-\b_5+\b_7) 1\ot {\d^a}_{\s(1)} \ot Y_b \ot {\d^b}_{\s(2)} = 0.
\end{align*}
Accordingly we get the system \eqref{system-vp1-Hochschild}.
\end{proof}
 \ni As a result we set $\b_j$, $1\le j\le 9$  to satisfy  \eqref{system-vp1-Hochschild}.

 \medskip

\ni On the next step we determine $\b_j$'s in such a way that $\psi
\in C^2_\Kc(\Ac,V)$ is a cyclic  cocycle over $(V_1\oplus V_2) \ot
\Ac^{\ot\,n}$.

\begin{lemma}\label{lemma-tauvp1=vp1}
We have $(\tau\psi)_1=\psi_1$ if and only if $\{\b_j\,|\,1 \leq j \leq 9\}$, satisfy
\begin{align}\label{aux-beta}
\begin{split}
& \b_1=\b_2=-r, \quad \b_3=\b_4=\b_5=-\b_6=\b_7=s, \\
& \b_8=-r-s, \quad \b_9 = \frac{1}{2}r+s.
\end{split}
\end{align}
\end{lemma}

\begin{proof}
By the Koszul coaction, we have
\begin{equation*}
t\psi(R^a \ot a_0\ot a_1 \ot a_2) = \psi_1(R^a \ot a_2 \ot a_0\ot a_1) - \psi_2(R^{ab} \ot Y_b(a_2) \ot a_0\ot a_1).
\end{equation*}
Accordingly,
\begin{align}
\begin{split}
& t\psi(R^a \ot a_0 \ot a_1 \ot a_2) = \\
& \b_1\tau\left( {\d^a}_{\s(1)}(a_0) X_{\s(2)}(a_1)a_2\right) + \b_2\tau\left( X_{\s(1)}(a_0) {\d^a}_{\s(2)}(a_1)a_2\right) \\
& + \b_3\tau\left( {\d^a}_{\s(1)}{\d^b}_{\s(2)}(a_0) Y_b(a_1)a_2\right) + \b_4\tau\left( Y_b(a_0) {\d^a}_{\s(1)}{\d^b}_{\s(2)}(a_1) a_2\right) \\
&  + \b_5\tau\left( {\d^a}_{\s(1)}Y_b(a_0) {\d^b}_{\s(2)}(a_1)a_2\right) + \b_6\tau\left( {\d^b}_{\s(1)}Y_b(a_0) {\d^a}_{\s(2)}(a_1)a_2\right) \\
& + \b_7\tau\left( {\d^a}_{\s(1)}(a_0) {\d^b}_{\s(2)}Y_b(a_1)a_2\right) + \b_8 \tau\left(  {\d^b}_{\s(1)}(a_0) {\d^a}_{\s(2)}Y_b(a_1)a_2\right) \\
&+ \b_9\tau\left( \D\left({\d^a}_{\s(1)\s(2)}\right)(a_0 \ot a_1)a_2\right) - r\tau\left( {\d^a}_{\s(1)}(a_0) {\d^b}_{\s(2)}(a_1) Y_b(a_2)\right) \\
&   - r\tau\left( {\d^b}_{\s(1)}(a_0) {\d^a}_{\s(2)}(a_1) Y_b(a_2)\right).
\end{split}
\end{align}
Hence $(t\psi)_1 = \psi_1$ if and only if
\begin{align*}
& \b_1\left(- 1 \ot {\d^a}_{\s(1)}X_{\s(2)} \ot 1 - 1 \ot X_{\s(2)} \ot {\d^a}_{\s(1)}\right) + \\
& \b_2\big(- 1 \ot X_{\s(1)} {\d^a}_{\s(2)} \ot 1 - 1 \ot {\d^a}_{\s(2)} \ot X_{\s(1)} + 1 \ot {\d^b}_{\s(1)} {\d^a}_{\s(2)} Y_b \ot 1 \\
& - 2 \cdot 1 \ot {\d^a}_{\s(1)\s(2)} \ot 1 + 1 \ot  {\d^a}_{\s(2)} \ot {\d^b}_{\s(1)}Y_b + 1 \ot {\d^a}_{\s(2)}Y_b \ot {\d^b}_{\s(1)} \\
& - 1 \ot \D\left({\d^a}_{\s(1)\s(1)}\right) \big) + \\
& \b_3\big(1 \ot {\d^a}_{\s(1)}{\d^b}_{\s(2)} Y_b \ot 1 + 1 \ot Y_b \ot {\d^a}_{\s(1)}{\d^b}_{\s(2)} \\
& + 1 \ot {\d^a}_{\s(1)} Y_b \ot {\d^b}_{\s(2)} + 1 \ot {\d^b}_{\s(2)} Y_b \ot {\d^a}_{\s(1)} \big) + \\
& \b_4\big(-1 \ot {\d^a}_{\s(1)}{\d^b}_{\s(2)}Y_b \ot 1 + 2 \cdot 1 \ot {\d^a}_{\s(1)\s(2)} \ot 1 -1 \ot {\d^a}_{\s(1)}{\d^b}_{\s(2)}\ot Y_b\big) + \\
& \b_5\big(1 \ot {\d^a}_{\s(1)}{\d^b}_{\s(2)}Y_b \ot 1 - 2 \cdot 1 \ot {\d^a}_{\s(1)\s(2)} \ot 1 + 1 \ot {\d^a}_{\s(1)}{\d^b}_{\s(2)}\ot Y_b+ \\
& + 1 \ot {\d^b}_{\s(2)}Y_b \ot {\d^a}_{\s(1)} + 1 \ot {\d^b}_{\s(2)} \ot {\d^a}_{\s(1)}Y_b + 1 \ot \D\left({\d^a}_{\s(1)\s(2)}\right) \big) + \\
& \b_6\big(1 \ot {\d^b}_{\s(1)}{\d^a}_{\s(2)}Y_b \ot 1 - 2 \cdot 1 \ot {\d^a}_{\s(1)\s(2)} \ot 1+ 1 \ot {\d^b}_{\s(1)}{\d^a}_{\s(2)}\ot Y_b\\
& + 1 \ot {\d^a}_{\s(2)}Y_b \ot {\d^b}_{\s(1)} - 1 \ot \D\left({\d^a}_{\s(1)\s(2)}\right) + 1 \ot {\d^a}_{\s(2)} \ot {\d^b}_{\s(1)}Y_b\big) +\\
& \b_7 \left(- 1 \ot {\d^a}_{\s(1)} {\d^b}_{\s(2)}Y_b \ot 1 - 1 \ot  {\d^b}_{\s(2)}Y_b \ot {\d^a}_{\s(1)}\right) +\\
& \b_8\left(- 1 \ot  {\d^b}_{\s(1)} {\d^a}_{\s(2)}Y_b \ot 1 -  1 \ot {\d^a}_{\s(2)}Y_b \ot {\d^b}_{\s(1)}\right) + \\
& \b_9\left(-2 \cdot 1 \ot {\d^a}_{\s(1)\s(2)} \ot 1 - 1 \ot \D\left({\d^a}_{\s(1)\s(2)}\right)\right) \\
&  + r1\ot  {\d^b}_{\s(2)} \ot {\d^a}_{\s(1)} Y_b + r 1 \ot {\d^a}_{\s(2)} \ot {\d^b}_{\s(1)} Y_b \\
& = \b_11\ot {\d^a}_{\s(1)} \ot X_{\s(2)} +\b_2 1\ot X_{\s(1)} \ot {\d^a}_{\s(2)} + \b_31\ot {\d^a}_{\s(1)}{\d^b}_{\s(2)} \ot Y_b \\
& \b_41\ot Y_b \ot {\d^a}_{\s(1)}{\d^b}_{\s(2)}  + \b_51\ot {\d^a}_{\s(1)}Y_b \ot {\d^b}_{\s(2)}+\b_6 1\ot {\d^b}_{\s(1)}Y_b \ot {\d^a}_{\s(2)} \\
&  + \b_71\ot {\d^a}_{\s(1)} \ot {\d^b}_{\s(2)}Y_b + \b_8 1\ot {\d^b}_{\s(1)} \ot {\d^a}_{\s(2)}Y_b + \b_9 1 \ot \D\left({\d^a}_{\s(1)\s(2)}\right).
\end{align*}
Collecting the terms, we obtain the equations
\begin{align}\label{system-vp1-cyclic}
\begin{split}
 \b_1-\b_2 & = 0 \\
 \b_2-\b_3+\b_4-\b_5+\b_6+\b_7-\b_8 &= 0 \\
 \b_1-2\b_2+2\b_4-2\b_5-2\b_6-2\b_9 &= 0 \\
 \b_2+\b_6+\b_7+k &=0 \\
 \b_2-\b_3+\b_5+\b_6-\b_8&=0 \\
 \b_5+\b_8+k &=0 \\
 \b_3+\b_5+\b_6-\b_7&=0\\
 \b_3-\b_4&=0\\
 -\b_2+\b_5-\b_6-2\b_9&=0.
\end{split}
\end{align}
Solving the systems \eqref{system-vp1-Hochschild} and \eqref{system-vp1-cyclic}
we obtain \eqref{aux-beta}.
\end{proof}
\ni As a result of Lemma \ref{system-vp1-Hochschild} we set $\b_j$,
$1\le j\le 9$ to satisfy \eqref{aux-beta}.

\medskip

\ni Finally we determine $\g_k$, $1\leq k \leq 4$ such that $\psi \in
C^2_\Kc(\Ac,V)$ is a Hochschild cocycle.

\begin{lemma}\label{lemma-bvp0=0}
We have $(b\psi)_0=0$ if and only if $\{\g_k\,|\,1 \leq k \leq 4\}$ satisfy
\begin{equation}\label{aux-gamma}
\g_1 = \g_2 = r, \quad \g_3 = \g_4 = s.
\end{equation}
\end{lemma}

\begin{proof}
Using the  Koszul coaction \eqref{aux-Koszul-K-coaction}, we have
\begin{align*}\label{aux-vp0-Hochschild}
& b\psi(1 \ot a_0\ot a_1 \ot a_2 \ot a_3) = \\
& \psi_0(1 \ot a_0a_1 \ot a_2 \ot a_3) - \psi_0(1 \ot a_0\ot a_1a_2 \ot a_3) \\
& + \psi_0(1 \ot a_0\ot a_1 \ot a_2a_3) - \psi_0(1 \ot a_3a_0\ot a_1 \ot a_2) \\
& + \psi_1(R^a \ot Y_a(a_3)a_0\ot a_1 \ot a_2) - \frac{1}{2!}\psi_2(R^{ab} \ot Y_bY_a(a_3)a_0\ot a_1 \ot a_2).
\end{align*}
As a result,  $(b\psi)_0 = 0$ if and only if
\begin{align*}
& \g_1 \Big(- 1\ot {\d^a}_{\s(1)} \ot Y_a \ot X_{\s(2)} + 1\ot X_{\s(1)} \ot {\d^a}_{\s(2)} \ot Y_a\Big) + \\
& \g_2 \Big(1\ot {\d^a}_{\s(1)} \ot X_{\s(2)} \ot Y_a + 1\ot {\d^a}_{\s(1)} \ot Y_a \ot  X_{\s(2)}  \\
& + 1\ot {\d^a}_{\s(1)} \ot {\d^b}_{\s(2)}Y_a \ot Y_b + 1\ot {\d^a}_{\s(1)} \ot {\d^b}_{\s(2)} \ot Y_bY_a\Big) + \\
& \g_3 \Big(- 1\ot {\d^a}_{\s(1)}{\d^b}_{\s(2)} \ot Y_b \ot Y_a  - 1\ot Y_b \ot {\d^a}_{\s(1)}{\d^b}_{\s(2)} \ot Y_a
\end{align*}
\begin{align}
\begin{split}
& - 1\ot {\d^a}_{\s(1)}Y_b \ot {\d^b}_{\s(2)} \ot Y_a - 1\ot {\d^a}_{\s(1)} \ot {\d^b}_{\s(2)}Y_b \ot Y_a \\
& - 1\ot {\d^b}_{\s(2)}Y_b \ot {\d^a}_{\s(1)} \ot Y_a - 1\ot {\d^b}_{\s(2)} \ot {\d^a}_{\s(1)}Y_b \ot Y_a \Big) +\\
& -\g_4 1 \ot \D({\d^a}_{\s(1)\s(2)}) \ot Y_a -r 1\ot {\d^a}_{\s(1)} \ot X_{\s(2)} \ot Y_a\\
&  + -r 1\ot X_{\s(1)} \ot {\d^a}_{\s(2)} \ot Y_a +s 1\ot {\d^a}_{\s(1)}{\d^b}_{\s(2)} \ot Y_b \ot Y_a \\
& + s 1\ot Y_b \ot {\d^a}_{\s(1)}{\d^b}_{\s(2)} \ot Y_a + s 1\ot {\d^a}_{\s(1)}Y_b \ot {\d^b}_{\s(2)} \ot Y_a\\
& -s 1\ot {\d^b}_{\s(1)}Y_b \ot {\d^a}_{\s(2)} \ot Y_a + s 1\ot {\d^a}_{\s(1)} \ot {\d^b}_{\s(2)}Y_b \ot Y_a \\
& + (-r-s) 1\ot {\d^b}_{\s(1)} \ot {\d^a}_{\s(2)}Y_b \ot Y_a + (\frac{r}{2}+s) 1 \ot \D\left({\d^a}_{\s(1)\s(2)}\right) \ot Y_a \\
& -r  1\ot {\d^a}_{\s(1)} \ot {\d^b}_{\s(2)} \ot Y_bY_a - \frac{r}{2} 1 \ot \D\left({\d^a}_{\s(1)\s(2)}\right) \ot Y_a = 0.
\end{split}
\end{align}
Hence we obtain \eqref{aux-gamma}.
\end{proof}

\begin{proposition}\label{prop-psi}
The cochain  $\psi: V \ot \Ac^{\ot\,3} \to \Cb$ is a
 cyclic 2-cocycle if and only if $\{\a_i\,|\,1 \leq i \leq 2\}$, satisfy \eqref{aux-alpha}, $\{\b_j\,|\,
1 \leq j \leq 9\}$, fulfill  \eqref{aux-beta},
   and $\{\g_k\,|\,1 \leq k \leq 4\}$ satisfy \eqref{aux-gamma}. The
   resulting  cocycle is then a SAYD-twisted cyclic cocycle.
\end{proposition}

\begin{proof}
We note that  $\psi$ is a Hochschild cocycle, \ie  $b\psi =
(b\psi)_0+(b\psi)_1+(b\psi)_2 =0$,  if and only if
  $(b\psi)_t =0$, $t=0,1,2$. We see that $(b\psi)_2=0$ via  Lemma
  \ref{lemma-bvp2=0},  $(b\psi)_1=0$ via  Lemma
  \ref{lemma-bvp1=0}, $(b\psi)_0=0$ via  Lemma
  \ref{lemma-bvp0=0}.

  \medskip

  \ni On the other hand $\psi$ is cyclic, \ie  $\tau\psi =\psi$,  if and only if
   $(\tau\psi)_t=\psi_t$, $t = 0,1,2$. Indeed, for $t=1$ Lemma \ref{lemma-tauvp1=vp1},
for $t=2$ Lemma \ref{lemma-tauvp2=vp2} yields the claims. As for $t=0$ we have
\begin{align*}
& t\psi(1 \ot a_0\ot a_1 \ot a_2) = \psi_0(1 \ot a_2 \ot a_0\ot a_1) - \psi_1(R^a \ot Y_a(a_2) \ot a_0\ot a_1) \\\notag
& +\frac{1}{2!}\psi_2(R^{ab} \ot Y_bY_a(a_2) \ot a_0\ot a_1).
\end{align*}
Accordingly,
\begin{align*}
& t\psi(1 \ot a_0\ot a_1 \ot a_2) =
\eqref{aux-box-1}+\eqref{aux-box-2}+
\eqref{aux-box-3}+\eqref{aux-box-4}.
\end{align*}
\FBOX{12}{
\begin{align}\label{aux-box-1}
\begin{split}
 &r\tau\Big(X_{\s(1)}(a_0)
X_{\s(2)}(a_1)a_2\Big) +
 r\tau\Big({\d^a}_{\s(1)} (a_0) X_{\s(2)}Y_a (a_1)a_2 \Big) \\
& + s\tau\Big({\d^a}_{\s(1)}\d^b_{\s(2)}Y_b(a_0) Y_a(a_1)a_2\Big) +
 s\tau\Big({\d^a}_{\s(1)\s(2)} (a_0)Y_a (a_1)a_2\Big)
\end{split}
 \end{align}
 }

\FBOX{14}{
 \begin{align}\label{aux-box-2}
 \begin{split}
&  r\tau\Big({\d^a}_{\s(1)}(a_0) X_{\s(2)}(a_1) Y_a(a_2)\Big) + r\tau\Big( X_{\s(1)} (a_0) {\d^a}_{\s(2)}(a_1) Y_a(a_2)\Big)\\
& - s\tau\Big({\d^a}_{\s(1)}\d^b_{\s(2)} (a_0) Y_b (a_1) Y_a(a_2)\Big) -s\tau\Big( Y_b (a_0) {\d^a}_{\s(1)}\d^b_{\s(2)} (a_1) Y_a(a_2)\Big) \\
&  +(r+s) \tau\Big(  \d^b_{\s(1)}(a_0) {\d^a}_{\s(2)}Y_b(a_1) Y_a(a_2) \Big)-s\tau\Big( {\d^a}_{\s(1)}Y_b(a_0) \d^b_{\s(2)} (a_1) Y_a (a_2)\Big)\\
& +s\tau\Big( \d^b_{\s(1)}Y_b (a_0) {\d^a}_{\s(2)} (a_1) Y_a (a_2))-s\tau\Big( {\d^a}_{\s(1)} (a_0) \d^b_{\s(2)}Y_b(a_1) Y_a(a_2)\Big) \\
\end{split}
\end{align}}

\FBOX{10}{
 \begin{equation}\label{aux-box-3}
(- \frac{r}{2}-s)\tau\Big(  \D\left({\d^a}_{\s(1)\s(2)}\right)(a_0\ot a_1) Y_a(a_2)\Big) \\
\end{equation}}

\FBOX{14}{
 \begin{equation}\label{aux-box-4}
 \frac{r}{2!} \tau\Big( {\d^a}_{\s(1)}(a_0) \d^b_{\s(2)}(a_1)
Y_bY_a(a_2) \Big)+ \frac{r}{2!} \tau\Big( \d^b_{\s(1)} (a_0)
{\d^a}_{\s(2)}(a_1) Y_bY_a(a_2)\Big)
\end{equation}}
We shall put the above expressions into the
 standard form $\tau(a_0h^1(a_1)h^2(a_2))$.  To this end we use
 the integration by parts property \eqref{aux-property-S-tilde}. On the computation below, we will employ the  actions
\begin{equation}
(f \ot g) \lt Y := f \ot g \cdot Y,\qquad (f \ot g) \btl Y := f  \cdot Y \ot g.
\end{equation}
of $\Fg_0$ on $\Hc_n^{\ot\,2}$. We first deal with \eqref{aux-box-3} by
\begin{align}\label{aux-sum-1}
\begin{split}
& \tau\Big( \D\left({\d^a}_{\s(1)\s(2)}\right)(a_0 \ot a_1) Y_a(a_2)\Big) =\\
& -2\tau\Big( a_0 {\d^a}_{\s(1)\s(2)}(a_1) Y_a(a_2)\Big) -\tau\Big( a_0 \left(\D\left({\d^a}_{\s(1)\s(2)}\right) \lt Y_a\right)(a_1 \ot a_2)\Big).
\end{split}
\end{align}
As for \eqref{aux-box-4} we compute
\begin{align}\label{aux-sum-2}
\begin{split}
& \frac{1}{2!} \tau\Big( {\d^a}_{\s(1)} (a_0) \d^b_{\s(2)} (a_1) Y_bY_a(a_2)\Big) + \frac{1}{2!} \tau\Big( \d^b_{\s(1)} (a_0) {\d^a}_{\s(2)} (a_1) Y_bY_a(a_2)\Big) = \\
& - \frac{1}{2} \tau\Big( a_0  \d^b_{\s(2)}(a_1) {\d^a}_{\s(1)}Y_bY_a(a_2) \Big)- \frac{1}{2} \tau\Big( a_0  {\d^a}_{\s(2)}(a_1) \d^b_{\s(1)}Y_bY_a(a_2) \Big)= \\
& \frac{1}{2} \tau\Big( a_0  {\d^a}_{\s(1)}(a_1) \d^b_{\s(2)}Y_bY_a (a_2)\Big)+ \frac{1}{2} \tau\Big( a_0  {\d^a}_{\s(1)} (a_1) \d^b_{\s(2)}Y_aY_b(a_2) \Big)= \\
& \tau\Big(a_0 {\d^a}_{\s(1)} (a_1) \d^b_{\s(2)}Y_bY_a(a_2)\Big) + \frac{1}{2} \tau\Big( a_0 {\d^a}_{\s(1)} (a_1) \d^b_{\s(2)}[Y_a,Y_b](a_2) \Big)= \\
& \tau\Big(a_0  {\d^a}_{\s(1)}(a_1) \d^b_{\s(2)}Y_bY_a (a_2)\Big)+
 \frac{1}{2} \tau\Big( a_0 \left(\D\left({\d^a}_{\s(1)\s(2)}\right) \lt Y_a\right)(a_1 \ot a_2)\Big).
\end{split}
\end{align}
On the third step we compute \eqref{aux-box-1} term by term.
\begin{align*}
& \tau\Big(X_{\s(1)} (a_0) X_{\s(2)}(a_1) a_2\Big) =\\
& -\tau\Big(a_0 X_{\s(2)} (a_1) X_{\s(1)}(a_2)\Big) +
\tau\Big(a_0 {\d^a}_{\s(1)}Y_aX_{\s(2)}(a_1) a_2 \Big) \\
&  + \tau\Big(a_0 X_{\s(2)}(a_1) {\d^a}_{\s(1)}Y_a(a_2)\Big)
 + \tau\Big(a_0 Y_aX_{\s(2)}(a_1) {\d^a}_{\s(1)}(a_2)\Big).
\end{align*}
\begin{align*}
&\tau\Big( {\d^a}_{\s(1)} (a_0) X_{\s(2)}Y_a (a_1) a_2\Big) =\\
& -\tau\Big(a_0 {\d^a}_{\s(1)} X_{\s(2)}Y_a (a_1)a_2 \Big)
 -\tau\Big(a_0  X_{\s(2)}Y_a (a_1) {\d^a}_{\s(1)}(a_2) \Big).
\end{align*}
\begin{align*}
& \tau\Big({\d^a}_{\s(1)}\d^b_{\s(2)}Y_b (a_0) Y_a (a_1) a_2\Big) =\\
& \tau\Big(Y_b (a_0) {\d^a}_{\s(1)}\d^b_{\s(2)}Y_a (a_1)a_2\Big) +\tau\Big( Y_b (a_0) Y_a (a_1) {\d^a}_{\s(1)}\d^b_{\s(2)}(a_2) \Big) \\
&  +\tau\Big( Y_b (a_0) {\d^a}_{\s(1)}Y_a (a_1) \d^b_{\s(2)} (a_2)\Big)+ \tau\Big(Y_b (a_0) \d^b_{\s(2)}Y_a (a_1) {\d^a}_{\s(1)}(a_2) \Big) =\\
&  -\tau\Big(a_0 {\d^a}_{\s(1)}\d^b_{\s(2)}Y_bY_a (a_1) a_2\Big) - 2\tau\Big( a_0 {\d^a}_{\s(1)\s(2)} Y_a (a_1) a_2 \Big)\\
&  -\tau\Big(a_0 {\d^a}_{\s(1)}\d^b_{\s(2)}Y_a (a_1) Y_b(a_2)\Big) -\tau\Big(a_0 Y_bY_a (a_1) {\d^a}_{\s(1)}\d^b_{\s(2)}(a_2) \Big)\\
&  -\tau\Big(a_0 Y_a (a_1) {\d^a}_{\s(1)}\d^b_{\s(2)}Y_b(a_2)\Big) -2 \tau\Big( a_0 Y_a (a_1) {\d^a}_{\s(1)\s(2)}(a_2) \Big)\\
&  -\tau\Big(a_0 {\d^a}_{\s(1)}Y_bY_a (a_1) \d^b_{\s(2)}(a_2)\Big) -\tau\Big( a_0 \left(\D\left({\d^a}_{\s(1)\s(2)}\right) \btl Y_a\right)(a_1 \ot a_2)\Big)\\
&  -\tau\Big(a_0{\d^a}_{\s(1)}Y_a(a_1) \d^b_{\s(2)}Y_b(a_2) \Big)- \tau\Big(a_0 \d^b_{\s(2)}Y_bY_a(a_1) {\d^a}_{\s(1)}(a_2) \Big)\\
&  -\tau\Big(a_0 \d^b_{\s(2)}Y_a (a_1) {\d^a}_{\s(1)}Y_b(a_2)\Big)
 -\tau\Big(a_0 \left(\D\left({\d^a}_{\s(1)\s(2)}\right) \btl Y_a\right)(a_1 \ot a_2) \Big).
\end{align*}
\begin{align*}
& \tau\Big({\d^a}_{\s(1)\s(2)} (a_0) Y_a(a_1) a_2\Big) =\\
&\tau\Big(a_0 {\d^a}_{\s(1)\s(2)} Y_a (a_1)a_2\Big)  + \tau\Big(a_0 Y_a (a_1)  {\d^a}_{\s(1)\s(2)}(a_2) \Big)\\
&  +\tau\Big( a_0\left(\D\left({\d^a}_{\s(1)\s(2)}\right) \btl Y_a\right)(a_1 \ot a_2)\Big).
\end{align*}
Summing up, we obtain
\begin{align}\label{aux-sum-3}
\begin{split}
& r\tau\Big(X_{\s(1)} (a_0) X_{\s(2)}(a_1) a_2\Big) +  r\tau\Big({\d^a}_{\s(1)} (a_0) X_{\s(2)}Y_a (a_1) a_2\Big)\\
& + s\tau\Big({\d^a}_{\s(1)}\d^b_{\s(2)}Y_b (a_0) Y_a (a_1)a_2\Big) + s\tau\Big({\d^a}_{\s(1)\s(2)}(a_0) Y_a (a_1)a_2\Big) = \\
& -r\tau\Big(a_0 X_{\s(2)} (a_1) X_{\s(1)}(a_2)\Big) -s\tau\Big(a_0 {\d^a}_{\s(1)}\d^b_{\s(2)}Y_a (a_1) Y_b(a_2)\Big) \\
&+r \tau\Big(a_0 X_{\s(2)}(a_1) {\d^a}_{\s(1)}Y_a(a_2) \Big) -s\tau\Big(a_0 Y_a (a_1) {\d^a}_{\s(1)}\d^b_{\s(2)}Y_b(a_2)\Big) \\
& -s\tau\Big(a_0 {\d^a}_{\s(1)}Y_a (a_1) \d^b_{\s(2)}Y_b(a_2)\Big)
 -s\tau\big(a_0 \d^b_{\s(2)}Y_a (a_1) {\d^a}_{\s(1)}Y_b(a_2)\Big).
\end{split}
\end{align}
Finally, for \eqref{aux-box-2} we have
\begin{align*}
& \tau\Big({\d^a}_{\s(1)}(a_0) X_{\s(2)}(a_1) Y_a(a_2)\Big) =\\
& -\tau\Big(a_0 {\d^a}_{\s(1)} X_{\s(2)}(a_1) Y_a(a_2)\Big)
 -\tau\big(a_0 X_{\s(2)}(a_1) {\d^a}_{\s(1)} Y_a(a_2) \Big),
\end{align*}
\begin{align*}
&\tau\Big( X_{\s(1)} (a_0) {\d^a}_{\s(2)} (a_1) Y_a(a_2)\Big) =\\
& -\tau\Big(a_0 X_{\s(1)} {\d^a}_{\s(2)}(a_1) Y_a(a_2)\Big) -\tau\Big(a_0  {\d^a}_{\s(2)} (a_1) X_{\s(1)}Y_a(a_2)\Big) \\
&  + \tau\Big(a_0 \d^b_{\s(1)} {\d^a}_{\s(2)}Y_b (a_1) Y_a \Big)- 2 \tau\Big( a_0 {\d^a}_{\s(1)\s(2)}(a_1) Y_a(a_2) \Big) \\
& + \tau\Big(a_0 {\d^a}_{\s(2)} (a_1) \d^b_{\s(1)}Y_bY_a(a_2) \Big) +\tau\Big(a_0 {\d^a}_{\s(2)}Y_b(a_1) \d^b_{\s(1)}Y_a (a_2)\Big) \\
&  -\tau\Big( a_0\left(\D\left({\d^a}_{\s(1)\s(2)}\right) \lt Y_a\right) (a_1\ot
a_2)\Big).
\end{align*}
\begin{align*}
&  \tau\Big({\d^a}_{\s(1)}\d^b_{\s(2)} (a_0) Y_b(a_1) Y_a(a_2)\Big) =\\
&  \tau\Big(a_0 {\d^a}_{\s(1)}\d^b_{\s(2)} Y_b (a_1) Y_a(a_2)\Big) + \tau\Big(a_0 Y_b (a_1) {\d^a}_{\s(1)}\d^b_{\s(2)}Y_a (a_2)\Big)\\
&  \tau\Big(a_0 {\d^a}_{\s(1)}Y_b (a_1) \d^b_{\s(2)}Y_a(a_2) \Big)+
 \tau\Big(a_0 \d^b_{\s(2)}Y_b (a_1) {\d^a}_{\s(1)}Y_a (a_2)\Big).
\end{align*}
%\begin{align*}
%& - Y_b \ot {\d^a}_{\s(1)}\d^b_{\s(2)} \ot Y_a = - Y_b \ot {\d^a}_{\s(1)}\d^b_{\s(2)} \ot Y_a \\
%\end{align*}
\begin{align*}
& \tau\Big( \d^b_{\s(1)}(a_0) {\d^a}_{\s(2)}Y_b (a_1) Y_a(a_2)\Big) =\\
&  \tau\Big( a_0 \d^b_{\s(1)} {\d^a}_{\s(2)}Y_b (a_1) Y_a (a_2)\Big)+\tau\Big(a_0 {\d^a}_{\s(2)}Y_b(a_1) \d^b_{\s(1)}Y_a(a_2)\Big).
\end{align*}
\begin{align*}
&  \tau\Big({\d^a}_{\s(1)}Y_b (a_0) \d^b_{\s(2)}(a_1) Y_a(a_2)\Big) =\\
% &\tau\Big( Y_b \ot {\d^a}_{\s(1)}\d^b_{\s(2)} \ot Y_a \Big)+\tau\Big( Y_b \ot \d^b_{\s(2)} \ot {\d^a}_{\s(1)}Y_a\Big) = \\
& -\tau\Big( Y_b (a_0) {\d^a}_{\s(1)}\d^b_{\s(2)} (a_1) Y_a(a_2) \Big)+ \tau\Big(a_0 \d^b_{\s(2)}Y_b (a_1) {\d^a}_{\s(1)}Y_a(a_2) \Big)\\
&  + \tau\Big(a_0\d^b_{\s(2)}(a_1)
{\d^a}_{\s(1)}Y_bY_a (a_2)\Big)+
 \tau\Big(a_0\left(\D\left({\d^a}_{\s(1)\s(2)}\right) \lt Y_a\right) (a_1\ot a_2)\Big).
\end{align*}
\begin{align*}
& \tau\Big(\d^b_{\s(1)}Y_b (a_0) {\d^a}_{\s(2)} (a_1) Y_a (a_2)\Big)=\\
%& - \tau\Big(Y_b \ot \d^b_{\s(1)}{\d^a}_{\s(2)} \ot Y_a\Big) - \tau\Big(Y_b \ot {\d^a}_{\s(2)} \ot \d^b_{\s(1)}Y_a \Big) = \\
&  \tau\Big( a_0 \d^b_{\s(1)}{\d^a}_{\s(2)}Y_b (a_1) Y_a(a_2)\Big) -2 \tau\Big( a_0 {\d^a}_{\s(1)\s(2)} (a_1) Y_a (a_2)\Big)\\
& +\tau\Big( a_0 \d^b_{\s(1)}{\d^a}_{\s(2)}(a_1) Y_bY_a(a_2)\Big) +\tau\Big( a_0{\d^a}_{\s(2)}Y_b(a_1) \d^b_{\s(1)}Y_a(a_2) \Big)\\
&  -\tau\Big( a_0 \left(\D\left({\d^a}_{\s(1)\s(2)}\right) \lt Y_a\right)(a_1 \ot a_2)\Big) + \tau\Big(a_0 {\d^a}_{\s(2)} (a_1) \d^b_{\s(1)} Y_bY_a (a_2)\Big).
\end{align*}
\begin{align*}
&  \tau\Big({\d^a}_{\s(1)}(a_0) \d^b_{\s(2)}Y_b (a_1) Y_a(a_2)\Big) =\\
& -\tau\Big(a_0 {\d^a}_{\s(1)} \d^b_{\s(2)}Y_b(a_1) Y_a(a_2)\big) - \tau\Big(a_0 \d^b_{\s(2)}Y_b(a_1) {\d^a}_{\s(1)}Y_a(a_2)\Big).
\end{align*}
Therefore,
\begin{align}\label{aux-sum-4}
\begin{split}
& r\tau\Big({\d^a}_{\s(1)} (a_0) X_{\s(2)} (a_1) Y_a(a_2) \Big)+ r\tau\Big( X_{\s(1)} (a_0) {\d^a}_{\s(2)} (a_1) Y_a (a_2)\Big) \\
& -s\tau\Big( {\d^a}_{\s(1)}\d^b_{\s(2)}(a_0) Y_b (a_1) Y_a(a_2) \Big) -s\tau\Big( Y_b (a_0) {\d^a}_{\s(1)}\d^b_{\s(2)} (a_1) Y_a(a_2)\Big) \\
& +(r+s) \tau\Big(  \d^b_{\s(1)} (a_0) {\d^a}_{\s(2)}Y_b (a_1) Y_a(a_2) \Big)-s\tau\Big( {\d^a}_{\s(1)}Y_b (a_0) \d^b_{\s(2)} (a_1) Y_a (a_2)\Big) \\
& +s\tau\Big( \d^b_{\s(1)}Y_b (a_0) {\d^a}_{\s(2)}(a_1) Y_a(a_2) \Big)-s\tau\Big( {\d^a}_{\s(1)}(a_0) \d^b_{\s(2)}Y_b (a_1) Y_a(a_2)\Big) \\
&= -r\tau\Big(a_0  {\d^a}_{\s(2)} (a_1) X_{\s(1)}Y_a(a_2)\Big) -r\tau\Big(a_0 X_{\s(2)} (a_1) {\d^a}_{\s(1)} Y_a(a_2)\Big)\\
& - s\tau\Big(a_0  Y_b (a_1) {\d^a}_{\s(1)}\d^b_{\s(2)}Y_a(a_2)\Big) - s\tau\Big(a_0  \d^b_{\s(2)}Y_b(a_1) {\d^a}_{\s(1)}Y_a(a_2)\Big)\\
& +s\tau\Big(a_0 {\d^a}_{\s(2)}Y_b (a_1) \d^b_{\s(1)}Y_a (a_2)\Big)- r\tau\Big(a_0 {\d^a}_{\s(1)}(a_1) \d^b_{\s(2)}Y_bY_a(a_2)\Big) \\
& (-r-s) \tau\Big( a_0 \left(\D\left({\d^a}_{\s(1)\s(2)}\right) \lt Y_a\right)(a_1 \ot a_2)\Big) +(-r-s) \tau\Big(a_0 {\d^a}_{\s(1)\s(2)}(a_1 ) Y_a(a_2)\Big).
\end{split}
\end{align}
As a result of the computations \eqref{aux-sum-1}, \eqref{aux-sum-2}, \eqref{aux-sum-3}, and \eqref{aux-sum-4} we obtain
\begin{equation*}
t\psi(1 \ot a_0\ot a_1 \ot a_2) = \psi_0(1 \ot a_0 \ot a_1\ot a_2).
\end{equation*}
\end{proof}

\begin{theorem}\label{theorem-vp}
The following cochain  $\vp=\vp_0+\vp_1+\vp_2\in C^2_\Kc(\Ac, V)$ is
a SAYD-twisted cyclic cocycle and  cohomologous to $\psi$ which is
defined in Proposition \ref{prop-psi}.
\begin{align}\label{aux-vp'-2-quick}
& \vp_2 = \chi_\tau^{\rm eq}\left(S_{ab}\ot\one \ot {\d^a}_{\s(1)} \ot {\d^b}_{\s(2)} + S_{ab}\ot\one \ot {\d^b}_{\s(1)} \ot {\d^a}_{\s(2)} \right) \\\notag
& \vp_1 = \chi_\tau^{\rm eq}\Big(- S_a\ot\one \ot {\d^a}_{\s(1)} \ot X_{\s(2)} - S_a\ot\one\ot X_{\s(1)} \ot {\d^a}_{\s(2)}\\\label{aux-vp'-1-quick}
 & - S_a\ot\one \ot {\d^b}_{\s(1)} \ot {\d^a}_{\s(2)}Y_b +\frac{1}{2}S_a\ot \one\ot \D({\d^a}_{\s(1)\s(2)})\Big) \\\label{aux-vp'-0-quick}
& \vp_0 = \chi_\tau^{\rm eq}\left(1^\ast\ot\one \ot X_{\s(1)} \ot X_{\s(2)} + 1^\ast\ot\one \ot {\d^a}_{\s(1)} \ot X_{\s(2)}Y_a\right),
\end{align}
\end{theorem}

\begin{proof}
As a result of Proposition \ref{prop-psi}  we can write $\psi = r\vp+s\phi$ for a 2-cochain $\phi = \phi_0+\phi_1+\phi_2$ given by
\begin{align}
& \phi_2 = 0 \\\notag
& \phi_1 = \chi_\tau^{\rm eq}\left(S_a\ot\one \ot {\d^a}_{\s(1)}{\d^b}_{\s(2)} \ot Y_b + S_a\ot\one\ot Y_b \ot {\d^a}_{\s(1)}{\d^b}_{\s(2)}\right. \\\notag
& + S_a\ot\one \ot {\d^a}_{\s(1)}Y_b \ot {\d^b}_{\s(2)} - S_a\ot\one \ot {\d^b}_{\s(1)}Y_b \ot {\d^a}_{\s(2)}  \\\notag
& + S_a\ot\one \ot {\d^a}_{\s(1)} \ot {\d^b}_{\s(2)}Y_b - S_a\ot\one \ot {\d^b}_{\s(1)} \ot {\d^a}_{\s(2)}Y_b \\
& +S_a\ot\one\ot \D({\d^a}_{\s(1)\s(2)})\Big) \\
& \phi_0 = \chi_\tau^{\rm eq}\left(1^\ast\ot\one \ot {\d^a}_{\s(1)}{\d^b}_{\s(2)}Y_b \ot Y_a +1^\ast\ot \one \ot {\d^a}_{\s(1)\s(2)} \ot Y_a\right).
\end{align}
We note that
\begin{equation}
\phi_1 = \chi_\tau^{\rm eq}\left(S_a\ot\one \ot \D\left({\d^a}_{\s(1)}{\d^b}_{\s(2)}Y_b\right) + S_a\ot\one\ot \D({\d^a}_{\s(1)\s(2)})\right).
\end{equation}
It is then straightforward to check that the 1-cochain $\phi'=\phi'_0+\phi'_1+\phi'_2$ given by
\begin{align}
& \phi'_2 = 0 \\
& \phi'_1 = \chi_\tau^{\rm eq}\left(S_a\ot\one \ot {\d^a}_{\s(1)}{\d^b}_{\s(2)}Y_b + S_a\ot\one \ot {\d^a}_{\s(1)\s(2)}\right) \\
& \phi'_0 = 0
\end{align}
is an equivariant cyclic 1-cocycle, and that
\begin{equation}
b\phi' = \phi.
\end{equation}
\end{proof}

\section{The characteristic map with coefficients}

In this section we construct a new characteristic map  from the
truncated Weil complex of the Lie algebra $\Fg_0:=g\ell_n$ to
 the cyclic complex of the crossed product algebra
 $\Ac = C_c^\infty(F^+)\rtimes \G$, and we
 illustrate  it completely  in codimensions $n=1$ and $n=2$. We
 observe that the resulting cocycles in codimension  1  match  with
 those   in \cite{ConnMosc98,ConnMosc} by Connes-Moscovici .

\medskip

\ni Such a characteristic map is obtained by composing a series of
maps
\begin{equation}
\xymatrix{ H(W(\Fg_0,V)) \ar[r]^{\FD_P} & H(C(\Fg_0,V))
\ar[r]^\cong& HC(U(\Fg_0),V) \ar[r]^{~~~\chi_\vp} & HC(\Ac). }
\end{equation}
As it is shown in \cite{RangSutl-II} the truncated Weil algebra is
identical with $W(\Fg_0,V)$. The Poincar\'e isomorphism
$\FD_{\rm P}$  is defined in \cite[Prop. 4.4]{RangSutl-II}. The middle
quasi-isomorphism is defined  in \cite[Thm. 6.2]{RangSutl-II}.
Finally the map $\chi_\vp$ is given by the cup product, in the sense of
\cite{KhalRang05-II,Rang08}, with the SAYD-twisted  cyclic cocycle
$\vp$ defined in Theorem \ref{theorem-vp}.

\medskip

\ni Let us recall the mentioned cup product from \cite{Rang08}. Let $C$
be a  $H$-module coalgebra and $A$ be an $H$-module algebra that
are equipped with a mapping
\begin{equation}
C \ot A \to A, \quad c \ot a \mapsto c(a)
\end{equation}
satisfying the conditions
\begin{equation}\label{aux-condition-C-A}
 (h \cdot c)(a) = h \cdot (c(a)),\quad  c(ab) =
c\ps{1}(a)c\ps{2}(b), \quad  c(1) = \ve(c)1.
\end{equation}
Let also $V$ be a  SAYD module over a Hopf algebra $H$.  One defines
\begin{equation}
\cup: C_H^p(C,V) \ot C_H^q(A,V) \to C^{p+q}(A)
\end{equation}
 for any $\vp \in C_H^q(A,V)$ and any $x = v \ot_H c^0 \odots
c^p \in C_H^p(C,V)$,
\begin{align}\label{aux-cup-product}
& (x \cup \vp)(a_0 \odots a_{p+q}) := \\\notag
 & \sum_{\s \in
{\rm Sh}(p,q)} (-1)^\s \p_{\wbar{\s}(p)} \ldots
\p_{\wbar{\s}(1)}\vp(\langle \p_{\wbar{\s}(p+q)} \ldots
\p_{\wbar{\s}(p+1)} x, a_0 \odots a_{p+q} \rangle),
\end{align}
where
\begin{equation}
\langle x, a_0 \odots a_n \rangle := v \ot_H c^0(a^0) \odots
c^n(a^n).
\end{equation}
Here ${\rm Sh}(p,q)$ is the set of all  $(p,q)$-shuffle
permutations, and
\begin{equation}
\wbar{\s}(n) = \s(n)-1.
\end{equation}
We set
$$C:= H := \Kc, \quad V= S(\Fg_0^\ast)\nsb{2n}, \quad A = \Ac.$$
Here $H$ acts on $C$ via multiplication, on $A$ as  \eqref{aux-action-Hn-on-A} , and
on $V$ via the coadjoint action.
 This construction yields for any  $\vp\in C^n_\Kc(\Ac,V)$ a characteristic map
\begin{equation}
\chi_\vp:C^\ell_\Kc(\Kc, V)\ra C^{\ell+n}(\Ac).
\end{equation}

\subsection{The characteristic map in codimension 1}
In this subsection we apply  the SAYD-twisted cyclic cocycle
\eqref{aux-pieces-of-cocycle} to illustrate the new characteristic
map
\begin{equation}
\chi_\vp:H(W(g\ell_1)\nsb{2})\longrightarrow
HC(C^\infty_c(F^+\Rb)\rtimes \G).
\end{equation}

\ni In order to verify that the new characteristic map is
geometrically meaningful,   we compare it with the Connes-Moscovici
computations for the classes of $\Hc_1$ in \cite{ConnMosc}. For the
convenience of the reader we recall the Hopf-cyclic classes in
$C(\Hc_1,\Cb_\d)$, namely the the transverse fundamental class
\begin{equation}
{\rm TF} = X \ot Y - Y \ot X - \d_1Y \ot Y \in C^2(\Hc_1, \Cb_\d)
\end{equation}
and the Godbillon-Vey class
\begin{equation}
{\rm GV} = \d_1 \in C^1(\Hc_1,\Cb_\d).
\end{equation}

\ni In view of the characteristic map
\eqref{aux-Connes-Moscovici-charac-map}, one  expresses  the
characteristic classes $TF \in C^2(\Ac_\Gamma)$ and $GV \in
C^1(\Ac_\Gamma)$ as
\begin{align}
& TF(a_0 \ot a_1 \ot a_2)  \\\notag
 & =\tau(a_0X(a_1)Y(a_2)) -
\tau(a_0Y(a_1)X(a_2)) - \tau(a_0\d_1Y(a_1)Y(a_2))
\end{align}
and
\begin{equation}
GV(a_0 \ot a_1)  = \tau(a_0\d_1(a_1)).
\end{equation}

\medskip

\ni In this subsection we set $\Kc = U(g\ell_1)$, $V = S(g\ell_1^\ast)\nsb{2}$, and $\Ac = C^\infty_c(F^+\Rb)\rtimes\G$.

\medskip

\ni The next step is to find the representative cocycles of
$H(g\ell_1,V)$. Let $\{Y\}$ and $\{\t\}$ be
 a dual pair of bases for $g\ell_1$ and $g\ell_1^\ast$.

\medskip

\ni By the Vey basis \cite{Godb72}, the cohomology of
$W(g\ell_1)\nsb{2}$ is spanned by
\begin{equation}
{\rm TF}:= 1\in S(g\ell_1^\ast)\nsb{2},\quad {\rm GV}:=\t\ot R\in
g\ell_1^\ast\ot S(g\ell_1^\ast)\nsb{2}
\end{equation}
Applying the Poincar\'e duality  \cite[Prop. 4.4]{RangSutl-II},  we
obtain the classes in $HC(g\ell_1, V)$
\begin{equation}
\FD_{\rm P}(1)= Y\ot 1\in g\ell_1\ot V, \quad \FD_{\rm P}( \t\ot R)=
R\in V.
\end{equation}

\begin{proposition}
The Hopf-cyclic cohomology $HC(\Kc,V)$ is generated by the classes
\begin{align}\label{aux-0-cocycle-codim-1}
& [R] \in HC^0(\Kc,V),\\\label{aux-1-cocycle-codim-1}
& [1 \ot Y + \frac{1}{2}R \ot Y^2] \in HC^1(\Kc,V).
\end{align}
\end{proposition}

\begin{proof}
We first check that $R \in C^0(\Kc,V)$ is a cyclic 0-cocycle. Indeed,
\begin{equation}
b(R) = R \ot 1 - R \ot 1 = 0, \qquad \Ft(R) = R.
\end{equation}
In a similar fashion,
\begin{equation*}
b(1 \ot Y) = 1 \ot 1 \ot Y - 1 \ot \D(Y) + 1 \ot Y \ot 1 + R \ot Y \ot Y
\end{equation*}
yields
\begin{equation}
b(1 \ot Y + \frac{1}{2}R \ot Y^2) = 0.
\end{equation}
Moreover, applying the cyclic map we observe
\begin{equation}
\Ft(1 \ot Y + \frac{1}{2}R \ot Y^2) = -1 \ot Y - R \ot Y^2 + \frac{1}{2}R \ot Y^2 = - (1 \ot Y + \frac{1}{2}R \ot Y^2).
\end{equation}
Finally, we apply the quasi-isomorphism
\begin{equation}\label{aux-quasi-inverse-of-anti-symm}
\mu:C^\bullet(U(\Fg),V) \longrightarrow C_\bullet(\Fg,V),
\end{equation}
which, for any Lie algebra $\Fg$, is the left inverse of the anti-symmetrization map, see \cite{ConnMosc,RangSutl-II}. Then we have
\begin{equation}
\mu(R) = R, \qquad \mu(1 \ot Y + \frac{1}{2}R \ot Y^2) = 1 \ot Y,
\end{equation}
the generators of the cohomology $HC(g\ell_1,V)$. This observation finishes the proof.
\end{proof}

\ni Let us find the images of the cocycles \eqref{aux-0-cocycle-codim-1} and \eqref{aux-1-cocycle-codim-1} under the map
\begin{equation}
\chi_\vp:C_\Kc^\bullet(\Kc,V) \to C^{\bullet+1}(\Ac).
\end{equation}
We compute
\begin{align}
\begin{split}
& \chi_\vp(\t)(a_0 \ot a_1) = (\vp \cup (\t \ot 1))(a_0 \ot a_1) \\
& = \vp(\langle\p_0(\t \ot 1),a_0 \ot a_1\rangle) = \vp(\t \ot a_0 \ot a_1) = -\tau(a_0\d_1(a_1)),
\end{split}
\end{align}
and in the same way,
\begin{align}
\begin{split}
& \chi_\vp(1 \ot Y + \frac{1}{2}\t \ot Y^2)(a_0 \ot a_1\ot a_2) = \\
& \sum_{\s \in Sh(1,1)} (-1)^\s \p_{\wbar{\s}(1)}\vp(\langle \p_{\wbar{\s}(2)}(1 \ot Y + \frac{1}{2}\t \ot Y^2),a_0 \ot a_1\ot a_2 \rangle) = \\
& -\p_0\vp(\langle \p_1(1 \ot Y + \frac{1}{2}\t \ot Y^2),a_0 \ot a_1\ot a_2 \rangle) + \\
& \p_1\vp(\langle \p_0(1 \ot Y + \frac{1}{2}\t \ot Y^2),a_0 \ot a_1\ot a_2 \rangle).
\end{split}
\end{align}
Then since
\begin{align}
\begin{split}
& \p_1(1 \ot Y + \frac{1}{2}\t \ot Y^2) = 1 \ot 1 \ot 1 \ot Y + 1 \ot 1 \ot Y \ot 1 +  \\
& \frac{1}{2} \t \ot 1 \ot Y^2 \ot 1 + \frac{1}{2}\t \ot 1 \ot 1 \ot Y^2 + \t \ot 1 \ot Y \ot Y,
\end{split}
\end{align}
we obtain
\begin{align}
\begin{split}
& \chi_\vp(1 \ot Y + \frac{1}{2}\t \ot Y^2)(a_0 \ot a_1\ot a_2) = \\
& \vp(-1 \ot a_0a_1 \ot Y(a_2) - 1 \ot a_0Y(a_1) \ot a_2 - \frac{1}{2}\t \ot a_0Y^2(a_1) \ot a_2 \\
& - \frac{1}{2}\t \ot a_0a_1 \ot Y^2(a_2) - \t \ot a_0Y(a_1) \ot Y(a_2)) + \vp(1 \ot a_0 \ot a_1Y(a_2) \\
& + \frac{1}{2}\t \ot a_0 \ot a_1Y^2(a_2)) \\
& = -\tau(a_0Y(a_1)X(a_2)) + \tau(a_0X(a_1)Y(a_2)) + \frac{1}{2}\tau(a_0Y^2(a_1)\d_1(a_2)) \\
& + \frac{1}{2}\tau(a_0\d_1(a_1)Y^2(a_2)) + \tau(a_0Y(a_1)\d_1Y(a_2)).
\end{split}
\end{align}

\begin{remark}\rm{
We note that
\begin{equation}
b(\d_1Y^2) \in C^2_{\Hc_1}(\Hc_1,\Cb_\d)
\end{equation}
is a cyclic cocycle and
\begin{equation}
\chi_\vp(1 \ot Y + \frac{1}{2}\t \ot Y^2) = \chi_\tau({\rm TF} + \frac{1}{2}b(\d_1Y^2)).
\end{equation}
Hence, we obtain the transverse fundamental class up to a coboundary. Similarly we have
\begin{equation}
\chi_\vp(\t) = - \chi_\tau({\rm GV}),
\end{equation}
\ie we also obtain the Godbillon-Vey class.
}\end{remark}

\subsection{The characteristic map in codimension 2}
In this subsection we exercise the machinery we developed in
Subsection \ref{subsection-H-n} for $n=2$. There is no such
computations in the literature that we know of.

\medskip

\ni We keep our conventions as before, \ie  $\Kc:= U(g\ell_2)$, $V
:= S(g\ell_2^\ast)\nsb{4}$, $\Hc:=\Kc_2$,  and $\Ac:=
C^\infty_c(F^+\Rb^2)\rtimes \G$.

\medskip

\ni Let us recall the Vey basis of the cohomology of the truncated
Weil algebra $W(g\ell_2)\nsb{4}$. To this end, we fix the following notation
 \begin{align*}
& c_1 = {\rm Tr} = R^1_1+R^2_2 \in S(g\ell_2^\ast), \qquad c_2 = R^1_2R^2_1 \in S(g\ell_2^\ast), \\
& u_1 = \t^1_1 + \t^2_2, \qquad u_2 = \t^1_1\wg\t^1_2\wg\t^2_1, \qquad \om = \t^1_1\wg\t^1_2\wg\t^2_1\wg\t^2_2.
\end{align*}
The  Vey basis \cite{Godb72}, is then introduced by
\begin{equation}
\big\{1,\, c_1^2\ot u_1,\,c_2\ot u_1,\,c_2\ot u_2,\, c_1^2 \ot
\om,\, c_2 \ot \om\big\}.
\end{equation}
Next, the Poincar\'e duality yields the 6 cocycles in the complex
$C(g\ell_2,V)$:
\begin{align*}
%\begin{split}
& \FD_{\rm P}(1) = 1 \ot Y_1^1\wg Y_1^2\wg Y_2^1\wg Y_2^2,\\
& \FD_{\rm P}(c_2 \ot u_1) =c_2 \ot \left(Y_1^2\wg Y_2^1\wg Y_2^2 - Y_1^1\wg Y_1^2\wg Y_2^1\right),\, \\
& \FD_{\rm P}(c_1^2 \ot u_1) = c_1^2 \ot \left(Y_1^2\wg Y_2^1\wg Y_2^2 - Y_1^1\wg Y_1^2\wg
Y_2^1\right) ,\\
& \FD_{\rm P}(c_2 \ot u_2)=c_2 \ot Y_2^2,\quad \FD_{\rm P}(c_1^2 \ot \om)=c_1^2,\\
&\FD_{\rm P}(c_2 \ot \om) = c_2 .
%\end{split}
\end{align*}
We label $Y_i^j$ as
\begin{equation}
Y_1:=Y_1^1, \;\; Y_2:= Y_1^2, \;\;Y_3:=Y_2^1,\;\; Y_4:=Y_2^2.
\end{equation}

\begin{proposition}
The Hopf-cyclic cohomology $HC(\Kc,V)$ is generated by the classes
\begin{align}
\begin{split}
& [\mathscr{TF}] \in HC^4(\Kc,V) \\
& [\mathscr{GV}] :=\Big[\sum_{\s \in S_3}(-1)^\s c_1^2 \ot \Big(Y_{\s(2)}\ot Y_{\s(3)}\ot Y_{\s(4)}\\
&\hspace{4cm} - Y_{\s(1)}\ot Y_{\s(2)}\ot Y_{\s(3)}\Big)\Big]\in HC^3(\Kc,V), \\
& [\mathscr{R}_1] := \Big[\sum_{\s \in S_3}(-1)^\s c_2 \ot \Big(Y_{\s(2)}\ot Y_{\s(3)}\ot Y_{\s(4)}\\
 &\hspace{4cm} - Y_{\s(1)}\ot Y_{\s(2)}\ot Y_{\s(3)}\Big)\Big]\in HC^3(\Kc,V), \\
& [\mathscr{R}_2] := [c_2 \ot Y_4] \in HC^1(\Kc,V),\\
& [\mathscr{R}_3] := [c_1^2] \in HC^0(\Kc,V),\\
& [\mathscr{R}_4] := [c_2] \in HC^0(\Kc,V) .
\end{split}
\end{align}
\end{proposition}

\begin{proof}
It is straightforward to check that $\mathscr{R}_1,\ldots,\mathscr{R}_4$ and $\mathscr{GV}$ are Hopf-cyclic cocycles and that
\begin{align*}
& \mu(\mathscr{R}_1) = \FD_{\rm P}(c_2\ot u_1),\quad \mu(\mathscr{R}_2) = \FD_{\rm P}(c_2\ot u_2),\\
& \mu(\mathscr{R}_3) = \FD_{\rm P}(c_1^2\ot \om), \quad  \mu(\mathscr{R}_4) = \FD_{\rm P}(c_2\ot \om),\\
& \mu(\mathscr{GV}) = \FD_{\rm P}(c_1^2\ot u_1).
\end{align*}
%Let us write the class in $HC_\Kc(\Kc,V)$ that corresponds to the class $\one \ot Y_1^1\wg Y_1^2\wg Y_2^1\wg Y_2^2 \in HC^4(g\ell_2,V)$ as
On the other hand,
\begin{equation}\label{aux-TF-in-K}
[\mathscr{TF}] = \Big[\sum_{\s \in S_4}(-1)^\s1 \ot Y_{\s(1)}\ot Y_{\s(2)}\ot Y_{\s(3)}\ot Y_{\s(4)}\Big] \in E_1^{2,2}(U(g\ell_2),V)
\end{equation}
is a cyclic cocycle  in the $E_1$ level of the spectral sequence
that corresponds to the natural filtration of $V$, \cite[Thm.
6.2]{RangSutl-II}.  Hence
\begin{equation*}
\mu(\mathscr{TF}) = \FD_{\rm P}(1).
\end{equation*}
Therefore, the claim follows from \cite[Thm. 6.2]{RangSutl-II}.
\end{proof}

\ni In this paper we do not complete the fundamental  cocycle as we
know its counterpart as a cyclic cocycle over $\Ac$ by the following
argument. Let us recall the characteristic  map
\begin{equation}
\chi_\vp:C^\bullet(\Kc,V) \longrightarrow C^{\bullet +2}(\Ac)
\end{equation}
for the SAYD-twisted cyclic 2-cocycle defined by the Theorem
\ref{theorem-vp}. To this end, we first prove a generalization of
\cite[Prop. 18]{ConnMosc}. In view of   in \cite{MoscRang09},
$\Hc_n$ is realized as a bicrossed product Hopf algebra
 $\Uc {\bcl}\Fc^{\rm cop}$. Here  $\Fc$
  is the commutative algebra of regular functions on the group of diffeomorphisms which
  preserve the  origin and with identity Jacobian at the origin,
  and $\Uc=U(g\ell_2^{\rm affine})$. The coaction involed in this
  bicrossed product realization  is recalled below
\begin{align}\label{aux-the-left-F^cop-coaction-on-U}
\begin{split}
& \nb:\Uc \to \Fc\cop \ot \Uc, \\
& X_k \mapsto 1 \ot X_k + \d^i_{jk} \ot Y_i^j,\qquad  Y_i^j \mapsto
1 \ot Y_i^j.
\end{split}
\end{align}
In the following proposition, for any $1 \leq j \leq m := n^2+n$,
\begin{equation}
\nb^j(Z) = Z\ns{-j} \odots Z\ns{-1} \ot Z\ns{0} \ot 1 \odots 1 \in \Hc_n^{\ot\,m+1}.
\end{equation}

\begin{proposition}
The $m:=n^2 + n$-cochain
\begin{equation}\label{aux-transverse-fundamental-class-in-Hn}
{\rm TF} := (-1)^{(m-1)!}\, \sum_{\s \in S_m} (-1)^\s \nb^m(Z^{\s(1)}) \cdots \nb(Z^{\s(m)}) \in \Hc_n^{\ot\,m+1}
\end{equation}
is a cyclic $m$-cocycle whose class $[{\rm TF}] \in HC^m(\Hc_n)$ corresponds, by the
Connes-Moscovici characteristic map, to the transverse fundamental class
$[TF] \in HC^m(\Ac)$.
\end{proposition}

\begin{proof}
Let  $a^i := f^iU^\ast_{\psi_i} \in \Ac$, where  $0 \leq i \leq m$,
$f^i\in C^\infty_c(F^+\Rb^2)$ and $\psi\in \G$. Without loss of
generality we assume that  $\psi_m\ldots\psi_0 = \Id$. The cyclic
cocycle $TF \in HC^m(\Ac)$ is given by the $m$-cocycle
\begin{align*}
%\begin{split}
& \displaystyle TF(a^0\odots a^m) = \int_{F^+\Rb^n} a^0 da^1\cdots
da^m, \qquad \qquad da^i=df^iU^\ast_{\psi_i}.
%\end{split}
\end{align*}
We note that in order to prove the claim, we need to find suitable
$h^0, \ldots, h^m \in \Hc_n$ such that
\begin{equation}
TF(a^0\odots a^m)=\tau(h^0(a^0)\cdots h^m(a^m)).
\end{equation}
Indeed,
\begin{align}\label{aux-h^0-to-h^m}
\begin{split}
&\int_{F^+\Rb^n} a^0 da^1\cdots da^m=
\int_{F^+\Rb^n}f^0{\psi_0}^\ast(df^1)
\cdots({\psi_0}^\ast\ldots{\psi_{m-1}}^\ast)(df^m) = \\
& \int_{F^+\Rb^n}h^0(f^0){\psi_0}^\ast(h^1(f^1))
\cdots({\psi_0}^\ast\ldots
{\psi_{m-1}}^\ast)(h^m(f^m)) \varpi = \\
& \int_{F^+\Rb^n}(\Id \ot {\psi_0}^\ast \odots
 {\psi_0}^\ast\ldots{\psi_{m-1}}^\ast)
 (h^0 \odots h^m)(f^0 \odots f^m)\varpi,
\end{split}
\end{align}
where the volume form on the frame bundle is
\begin{equation}\label{aux-volume-form}
\varpi = \bigwedge_{i = 1}^n\t^i \wg \bigwedge_{1 \leq i,j \leq
n}\om^i_j
 \quad {\rm (ordered \,\, lexicographically)}
\end{equation}
In the above computation  we use the following notations.
\begin{equation*}
(h^0 \odots h^m)(f^0 \odots f^m) = h^0(f^0)\ldots h^m(f^m),
\end{equation*}
and similarly for any $g^0,\ldots,g^m \in C^\infty_c(F^+\Rb^n)$,
\begin{align*}
(\Id \ot {\psi_0}^\ast \odots
{\psi_0}^\ast\ldots{\psi_{m-1}}^\ast)&(g^0 \odots g^m)\\
&=g^0{\psi_0}^\ast(g^1)\ldots
{\psi_0}^\ast\ldots{\psi_{m-1}}^\ast(g^m).
\end{align*}
Here  $\psi^\ast(g)(x,y)= g(\psi(x), \psi'(x)\cdot y)$.
\medskip

\ni For  any $f \in C^\infty_c(F^+\Rb^n)$ we have
\begin{equation}
df = \frac{\p f}{\p x^i}dx^i + \frac{\p f}{\p y_i^j}dy_i^j =  X_i(f)\t^i + Y_i^j(f)\om^i_j.
\end{equation}
Therefore, for $\psi_0 \in \Diff(\Rb^n)$, and $f^0, f^1 \in C^\infty_c(F^+\Rb^n)$ we have
\begin{align}\label{aux-f^0-df^1-formula}
\begin{split}
& f^0{\psi_0}^\ast (df^1) =f^0{\psi_0}^\ast(X_i(f^1)){\psi_0}^\ast(\t^i) + f^0{\psi_0}^\ast(Y_i^j(f^1)){\psi_0}^\ast(\om^i_j) \\
& = f^0{\psi_0}^\ast(X_i(f^1))\t^i + f^0{\psi_0}^\ast(Y_i^j(f^1))(\g^i_{jk}(\psi_0)\t^k + \om^i_j) \\
& = (\Id \ot {\psi_0}^\ast) \big[(1 \ot X_k + \d^i_{jk} \ot Y_i^j)(f^0 \ot f^1)\t^k + (1 \ot Y_i^j)(f^0 \ot f^1)\om^i_j\big] \\
& = (\Id \ot {\psi_0}^\ast) ((X_k)\ns{-1} \ot (X_k)\ns{0})(f^0 \ot f^1)\t^k + \\
& (\Id \ot {\psi_0}^\ast)((Y_i^j)\ns{-1} \ot (Y_i^j)\ns{0})(f^0 \ot
f^1)\om^i_j.
\end{split}
\end{align}
\ni On the second equality we have used \cite[(2.16)]{ConnMosc}, and on the third equality we used \eqref{aux-action-Hn-on-A}. On the forth equality, the left coaction is \eqref{aux-the-left-F^cop-coaction-on-U}.

\medskip

\ni On the other hand we have
\begin{align}\label{aux-multiplication-f^0-df^1-to-df^m}
\begin{split}
& (-1)^{(m-1)!} \,\, f^0{\psi_0}^\ast (df^1) {\psi_0}^\ast{\psi_1}^\ast (df^2) \ldots {\psi_0}^\ast \cdots {\psi_{m-1}}^\ast (df^m)  \\
&= {\psi_0}^\ast \cdots {\psi_{m-1}}^\ast (df^m) \ldots {\psi_0}^\ast{\psi_1}^\ast (df^2) {\psi_0}^\ast (df^1) f^0 \\
& =\left({\psi_{m-1} \cdots \psi_0}^\ast(X_i(f^m))\t^i \right. +\\
& \left. {\psi_{m-1} \cdots \psi_0}^\ast(Y_i^j(f^m))(\g^i_{jk}(\psi_{m-1} \cdots \psi_1)\t^k + \om^i_j)\right) \\
& \cdots \\
& \left({\psi_1\psi_0}^\ast(X_i(f^2))\t^i + {\psi_1\psi_0}^\ast(Y_i^j(f^2))(\g^i_{jk}(\psi_2\psi_1)\t^k + \om^i_j)\right) \cdot \\
& \left( {\psi_0}^\ast(X_i(f^1))\t^i + {\psi_0}^\ast(Y_i^j(f^1))(\g^i_{jk}(\psi_1)\t^k + \om^i_j)\right) \cdot f^0  \\
& =\left(\Id \ot {\psi_0}^\ast \odots {\psi_0}^\ast\ldots{\psi_{m-1}}^\ast\right) \Big\{ \\
& \left[((X_i)\ns{-m} \odots (X_i)\ns{0})\t^i + ((Y_i^j)\ns{-m} \odots (Y_i^j)\ns{0})\om^i_j\right] \cdot \\
& \cdots\\
& \big[((X_i)\ns{-2} \ot (X_i)\ns{-1} \ot (X_i)\ns{0} \ot 1 \odots 1)\t^i + \\
 & \hspace{3cm}((Y_i^j)\ns{-2} \ot (Y_i^j)\ns{-1} \ot (Y_i^j)\ns{0} \ot 1 \odots 1)\om^i_j\big] \cdot \\
 & \left[((X_i)\ns{-1} \ot (X_i)\ns{0} \ot 1 \odots 1)\t^i\right. + \\
& \left. ((Y_i^j)\ns{-1} \ot (Y_i^j)\ns{0} \ot 1 \odots 1)\om^i_j\right] \Big\} \left(f^0 \odots f^m\right).
\end{split}
\end{align}
\ni On the third equality, we have used the cocycle identity \cite[(1.16)]{MoscRang09} in order to obtain the expressions in $\Hc_n^{\ot\,m+1}$ in the range of the coaction \eqref{aux-the-left-F^cop-coaction-on-U}. We reversed the order of the multiplication in \eqref{aux-multiplication-f^0-df^1-to-df^m} in order to avoid obtaining elements in $\Hc_n^{\ot\,m+1}$ involving $Y_i^j\d^p_{qr} \in \Hc_n$ which do not belong the PBW basis of $\Hc_n$, \cite[Prop. 3]{ConnMosc}.

\medskip

\ni The coefficient of the volume form \eqref{aux-volume-form}, which is an element $H \in \Hc_n^{\ot\,m+1}$, can now be expressed by carrying out the multiplication in \eqref{aux-multiplication-f^0-df^1-to-df^m}. Let $(Z^1, \ldots, Z^m) = (X_1,\ldots, X_n, Y_1^1, \ldots, Y_n^n)$, where the right hand side is ordered lexicographically. Then
\begin{align}\label{aux-the-element-H}
\begin{split}
& H = \sum_{\s \in S_m} (-1)^\s \nb^m(Z^{\s(1)}) \cdots \nb(Z^{\s(m)}).
\end{split}
\end{align}

\medskip

\ni Finally, as a result of \eqref{aux-h^0-to-h^m},
\eqref{aux-multiplication-f^0-df^1-to-df^m} and \eqref{aux-the-element-H} we have the element
\begin{equation}
{\rm TF} := (-1)^{(m-1)!}\, \sum_{\s \in S_m} (-1)^\s \nb^m(Z^{\s(1)}) \cdots \nb(Z^{\s(m)}) \in \Hc_n^{\ot\,m+1}
\end{equation}
such that for the Connes-Moscovici characteristic map \eqref{aux-Connes-Moscovici-charac-map} we have
\begin{equation}
\chi_\tau({\rm TF}) = TF \in C^m(\Ac).
\end{equation}
\end{proof}
\ni Let us illustrate the proposition for $n=1$. We have $(Z^1,Z^2) = (X,Y)$, $m = 1^2+1=2$ and
\begin{align}
\begin{split}
& {\rm TF} = (-1)^{1!}\sum_{\s \in S_2}(-1)^\s \nb^2(Z^{\s(1)})\nb (Z^{\s(2)}) \\
& = -\Big((X\ns{-2} \ot X\ns{-1} \ot X\ns{0})(Y\ns{-1} \ot Y\ns{0} \ot 1) - \\
& (Y\ns{-2} \ot Y\ns{-1} \ot Y\ns{0})(X\ns{-1} \ot X\ns{0} \ot 1)\Big) \\
& =  - 1 \ot Y \ot X - 1\ot \d_1Y \ot Y - \d_1 \ot Y \ot Y + 1 \ot X \ot Y + \d_1 \ot Y \ot Y \\
& = 1 \ot X \ot Y- 1 \ot Y \ot X - 1\ot \d_1Y \ot Y.
\end{split}
\end{align}
Next we recall the isomorphism
\begin{align}\label{aux-diagonal-isomorphism}
\begin{split}
& \Psi_{\bcl}:C^\bullet(\Hc, \Cb_\d) \to D^\bullet(\Uc,\Fc,\Cb_\d) \\
&\Psi_{\bcl}( u^1\acl f^1\ot \ldots\ot u^n\acl f^n) = \\
& u^1\ns{-n}f^1 \odots u^1\ns{-1}\ldots u^n\ns{-1}f^n \ot u^1\ns{0} \odots u^n\ns{0}
\end{split}
\end{align}
defined in \cite{MoscRang09} that identifies the Hopf-cyclic complex $C^\bullet(\Hc,\Cb_\d)$
of a Hopf algebra $\Hc = \Uc \bcl \Fc$ with the diagonal subcomplex
\begin{equation}
D^\bullet(\Uc,\Fc,\Cb_\d) := \Cb_\d \ot \Fc^{\ot \, \bullet} \ot \Uc^{\ot\,\bullet}.
\end{equation}
\ni On the other hand, for $\Uc = U(\Fg)$ we have the quasi-isomorphism
\begin{equation}\label{aux-mu-map}
\mu:\Cb_\d \ot \Fc^{\ot \, \bullet} \ot \Uc^{\ot\,\bullet} \to \Cb_\d \ot \Fc^{\ot \, \bullet} \ot \bigwedge^{\bullet}\Fg,
\end{equation}
which is the inverse of the antisymmetrization map on the level of cohomologies.

\begin{remark}\rm{
The transverse fundamental class $[{\rm TF}] \in HC^{n^2+n}(\Hc_n,\Cb_\d)$ defined in \eqref{aux-transverse-fundamental-class-in-Hn} corresponds to the class
\begin{equation}
[1 \ot X_1 \wdots X_n \wg Y_1^1 \wdots Y_n^n],
\end{equation}
in the total complex  $C^{\bullet, \bullet}(\Fg, \Fc, \Cb_\d)$ \cite[(3.37)]{MoscRang09}, by the composition of \eqref{aux-diagonal-isomorphism} and \eqref{aux-mu-map}.
}\end{remark}
\ni On the next move, we introduce the commutative diagram
\begin{equation}\label{aux-diagram-of-characteristic-maps}
\xymatrix {
\ar[dr]_{\chi_\vp} C_\Kc^j(\Kc,V) \ar[r]^{\wbar{\chi}_\vp} & C^{j+k}(\Hc_n,\Cb_\d) \ar[d]^{T^\natural} \\
& C^{j+k}(\Ac)
}
\end{equation}
induced by (a decomposition of) the cup product \eqref{aux-cup-product} via a cyclic cocycle  $\vp \in C^k_{\Kc}(\Ac,V)$ in the image of \eqref{aux-equivariant-map}. Here $T^\natural:\Hc_n^j \to C^j(\Ac)$ is the isomorphism
\begin{equation}
T^\natural(h^1 \odots h^j)(a^0 \odots a^j) = \tau(a^0h^1(a^1)\ldots h^j(a^j)),
\end{equation}
defined in \cite[(3.12)]{ConnMosc}, onto the space of elementary characteristic $j$-cochains, \cite[Section 3]{ConnMosc}.

\medskip

\ni We are now ready to prove our claim. On the following proposition, $\vp \in C^2_{\Kc}(\Ac,V)$ is the cyclic cocycle defined in \eqref{aux-vp'-2-quick},\eqref{aux-vp'-1-quick}, \eqref{aux-vp'-0-quick}.

\begin{proposition}
The cyclic cohomology class $[\mathscr{TF}] \in HC^{4}(\Kc,V)$ is mapped by $\chi_\vp:HC^{4}(\Kc,V) \to HC^6(\Ac_\G)$ to the transverse characteristic class $[TF] \in HC^6(\Ac_\G)$.
\end{proposition}

\begin{proof}
By the diagram \eqref{aux-diagram-of-characteristic-maps} we understand that it is enough to observe $[\wbar{\chi}_\vp(\mathscr{TF})] = [{\rm TF}] \in HC^6(\Hc_n)$. This, in turn, follows from the observation
\begin{equation}
\mu\circ \psi_{\bcl}([\wbar{\chi}_\vp(\mathscr{TF})]) = \mu\circ\psi_{\bcl}([{\rm TF}]) = [1 \ot X_1 \wg X_2 \wg Y_1^1 \wdots Y_2^2],
\end{equation}
 thanks to the large kernel of \eqref{aux-mu-map}. Hence the result follows since $\mu\circ \psi_{\bcl}$ is an isomorphism on the level of cohomologies.
\end{proof}

\ni In the following we  present the image of the cyclic cocycles  $\mathscr{GV}$, $\mathscr{R}_1$, $\mathscr{R}_2$, $\mathscr{R}_3$,
$\mathscr{R}_4$ under the characteristic map  $\chi_\vp: C^\bullet(\Kc, V)\ra C^{\bullet+2}(\Ac)$. These are cyclic  cocycles in $C^\bullet(\Ac)$.  We do not display the detailed account of the  computation as it is  lengthy and straightforward.

\begin{align}
\begin{split}
& \chi_\vp(\mathscr{GV})(a_0 \odots a_5) = \sum_{k=1}^3\sum_{1 \leq i,j \leq 2}\sum_{\s,\g,\eta \in S_2}2 \cdot (-1)^\s(-1)^\g(-1)^{k-1} \\
&\Big\{-\tau\left(a_0\d^i_{i\eta(1)}(a_1)\d^j_{j\eta(2)}(a_2) Y_{\mu^k(\s(1))}^{\mu^k(\s(2))}(a_3) Y_{\mu^k(\g(1))}^{\mu^k(\s(1))}(a_4)  Y_{\mu^k(\g(2))}^{\mu^k(\s(1))}(a_5)\right) \\
& + \tau\left(a_0\d^i_{i\eta(1)}(a_1) Y_{\mu^k(\s(1))}^{\mu^k(\s(2))}(a_2) \d^j_{j\eta(2)}(a_3) Y_{\mu^k(\g(1))}^{\mu^k(\s(1))}(a_4)  Y_{\mu^k(\g(2))}^{\mu^k(\s(1))}(a_5)\right) \\
&-\tau\left(a_0\d^i_{i\eta(1)}(a_1) Y_{\mu^k(\s(1))}^{\mu^k(\s(2))}(a_2) Y_{\mu^k(\g(1))}^{\mu^k(\s(1))}(a_3)  \d^j_{j\eta(2)}(a_4)Y_{\mu^k(\g(2))}^{\mu^k(\s(1))}(a_5)\right) \\
&+\tau\left(a_0\d^i_{i\eta(1)}(a_1) Y_{\mu^k(\s(1))}^{\mu^k(\s(2))}(a_2) Y_{\mu^k(\g(1))}^{\mu^k(\s(1))}(a_3)  Y_{\mu^k(\g(2))}^{\mu^k(\s(1))}(a_4)\d^j_{j\eta(2)}(a_5)\right) \\
&-\tau\left(a_0 Y_{\mu^k(\s(1))}^{\mu^k(\s(2))}(a_1)\d^i_{i\eta(1)}(a_2)\d^j_{j\eta(2)}(a_3) Y_{\mu^k(\g(1))}^{\mu^k(\s(1))}(a_4)  Y_{\mu^k(\g(2))}^{\mu^k(\s(1))}(a_5)\right) \\
&+\tau\left(a_0 Y_{\mu^k(\s(1))}^{\mu^k(\s(2))}(a_1)\d^i_{i\eta(1)}(a_2) Y_{\mu^k(\g(1))}^{\mu^k(\s(1))}(a_3)  \d^j_{j\eta(2)}(a_4)Y_{\mu^k(\g(2))}^{\mu^k(\s(1))}(a_5)\right) \\
&-\tau\left(a_0 Y_{\mu^k(\s(1))}^{\mu^k(\s(2))}(a_1)\d^i_{i\eta(1)}(a_2) Y_{\mu^k(\g(1))}^{\mu^k(\s(1))}(a_3)  Y_{\mu^k(\g(2))}^{\mu^k(\s(1))}(a_4)\d^j_{j\eta(2)}(a_5)\right) \\
&-\tau\left(a_0 Y_{\mu^k(\s(1))}^{\mu^k(\s(2))}(a_1) Y_{\mu^k(\g(1))}^{\mu^k(\s(1))}(a_2)\d^i_{i\eta(1)}(a_3)  \d^j_{j\eta(2)}(a_4)Y_{\mu^k(\g(2))}^{\mu^k(\s(1))}(a_5)\right) \\
&+\tau\left(a_0 Y_{\mu^k(\s(1))}^{\mu^k(\s(2))}(a_1) Y_{\mu^k(\g(1))}^{\mu^k(\s(1))}(a_2)\d^i_{i\eta(1)}(a_3)  Y_{\mu^k(\g(2))}^{\mu^k(\s(1))}(a_4)\d^j_{j\eta(2)}(a_5)\right) \\
&-\tau\left(a_0 Y_{\mu^k(\s(1))}^{\mu^k(\s(2))}(a_1) Y_{\mu^k(\g(1))}^{\mu^k(\s(1))}(a_2)  Y_{\mu^k(\g(2))}^{\mu^k(\s(1))}(a_3)\d^i_{i\eta(1)}(a_4)\d^j_{j\eta(2)}(a_5)\right) \Big\}.
\end{split}
\end{align}

\begin{align}\notag
%\begin{split}
& \chi_\vp(\mathscr{R}_2)(a_0 \odots a_3)\\\notag
& = \sum_{\s \in S_2}(-1)^\s \Big\{-\tau(a_0\d^1_{2\s(1)}(a_1)\d^2_{1\s(2)}(a_2)Y_2^2(a_3)) \\\notag
&- \tau(a_0\d^2_{1\s(1)}(a_1)\d^1_{2\s(2)}(a_2)Y_2^2(a_3)) +\tau(a_0\d^1_{2\s(1)}(a_1)Y_2^2(a_2)\d^2_{1\s(2)}(a_3))\\\notag
& + \tau(a_0\d^2_{1\s(1)}(a_1)Y_2^2(a_2)\d^1_{2\s(2)}(a_3))-\tau(a_0Y_2^2(a_1)\d^1_{2\s(1)}(a_2)\d^2_{1\s(2)}(a_3)) \\
& -\tau(a_0Y_2^2(a_1)\d^2_{1\s(1)}(a_2)\d^1_{2\s(2)}(a_3))\Big\}.
%\end{split}
\end{align}
\begin{align}
 \chi_\vp(\mathscr{R}_3)(a_0 \ot a_1 \ot a_2) = \sum_{1 \leq i,j \leq 2}\sum_{\s \in S_2}2 \cdot (-1)^\s \tau(a_0\d^i_{i\s(1)}(a_1)\d^j_{j\s(2)}(a_2)).
\end{align}

\begin{align}
\begin{split}
& \chi_\vp(\mathscr{R}_1)(a_0 \odots a_5) = \sum_{k=1}^3\sum_{\s,\g,\eta \in S_2}(-1)^\s(-1)^\g(-1)^{k-1} \Big\{\\
&-\tau\left(a_0\d^1_{2\eta(1)}(a_1)\d^2_{1\eta(2)}(a_2) Y_{\mu^k(\s(1))}^{\mu^k(\s(2))}(a_3) Y_{\mu^k(\g(1))}^{\mu^k(\s(1))}(a_4)  Y_{\mu^k(\g(2))}^{\mu^k(\s(1))}(a_5)\right) \\
&-\tau\left(a_0\d^2_{1\eta(1)}(a_1)\d^1_{2\eta(2)}(a_2) Y_{\mu^k(\s(1))}^{\mu^k(\s(2))}(a_3) Y_{\mu^k(\g(1))}^{\mu^k(\s(1))}(a_4)  Y_{\mu^k(\g(2))}^{\mu^k(\s(1))}(a_5)\right) \\
& + \tau\left(a_0\d^1_{2\eta(1)}(a_1) Y_{\mu^k(\s(1))}^{\mu^k(\s(2))}(a_2) \d^2_{1\eta(2)}(a_3) Y_{\mu^k(\g(1))}^{\mu^k(\s(1))}(a_4)  Y_{\mu^k(\g(2))}^{\mu^k(\s(1))}(a_5)\right) \\
& + \tau\left(a_0\d^2_{1\eta(1)}(a_1) Y_{\mu^k(\s(1))}^{\mu^k(\s(2))}(a_2) \d^1_{2\eta(2)}(a_3) Y_{\mu^k(\g(1))}^{\mu^k(\s(1))}(a_4)  Y_{\mu^k(\g(2))}^{\mu^k(\s(1))}(a_5)\right) \\
&-\tau\left(a_0\d^1_{2\eta(1)}(a_1) Y_{\mu^k(\s(1))}^{\mu^k(\s(2))}(a_2) Y_{\mu^k(\g(1))}^{\mu^k(\s(1))}(a_3)  \d^2_{1\eta(2)}(a_4)Y_{\mu^k(\g(2))}^{\mu^k(\s(1))}(a_5)\right) \\
&-\tau\left(a_0\d^2_{1\eta(1)}(a_1)Y_{\mu^k(\s(1))}^{\mu^k(\s(2))}(a_2) Y_{\mu^k(\g(1))}^{\mu^k(\s(1))}(a_3)  \d^1_{2\eta(2)}(a_4) Y_{\mu^k(\g(2))}^{\mu^k(\s(1))}(a_5)\right) \\
&+\tau\left(a_0\d^1_{2\eta(1)}(a_1) Y_{\mu^k(\s(1))}^{\mu^k(\s(2))}(a_2) Y_{\mu^k(\g(1))}^{\mu^k(\s(1))}(a_3)  Y_{\mu^k(\g(2))}^{\mu^k(\s(1))}(a_4)\d^2_{1\eta(2)}(a_5)\right) \\
&+\tau\left(a_0\d^2_{1\eta(1)}(a_1)Y_{\mu^k(\s(1))}^{\mu^k(\s(2))}(a_2) Y_{\mu^k(\g(1))}^{\mu^k(\s(1))}(a_3)  Y_{\mu^k(\g(2))}^{\mu^k(\s(1))}(a_4)\d^1_{2\eta(2)}(a_5)\right)\\
&-\tau\left(a_0 Y_{\mu^k(\s(1))}^{\mu^k(\s(2))}(a_1)\d^1_{2\eta(1)}(a_2)\d^2_{1\eta(2)}(a_3) Y_{\mu^k(\g(1))}^{\mu^k(\s(1))}(a_4)  Y_{\mu^k(\g(2))}^{\mu^k(\s(1))}(a_5)\right) \\
&-\tau\left(a_0Y_{\mu^k(\s(1))}^{\mu^k(\s(2))}(a_1)\d^2_{1\eta(1)}(a_2) \d^1_{2\eta(2)}(a_3)Y_{\mu^k(\g(1))}^{\mu^k(\s(1))}(a_4)  Y_{\mu^k(\g(2))}^{\mu^k(\s(1))}(a_5)\right) \\
&+\tau\left(a_0 Y_{\mu^k(\s(1))}^{\mu^k(\s(2))}(a_1)\d^1_{2\eta(1)}(a_2) Y_{\mu^k(\g(1))}^{\mu^k(\s(1))}(a_3)  \d^2_{1\eta(2)}(a_4)Y_{\mu^k(\g(2))}^{\mu^k(\s(1))}(a_5)\right) \\
&+\tau\left(a_0Y_{\mu^k(\s(1))}^{\mu^k(\s(2))}(a_1)\d^2_{1\eta(1)}(a_2) Y_{\mu^k(\g(1))}^{\mu^k(\s(1))}(a_3)  \d^1_{2\eta(2)}(a_4)Y_{\mu^k(\g(2))}^{\mu^k(\s(1))}(a_5)\right) \\
&-\tau\left(a_0 Y_{\mu^k(\s(1))}^{\mu^k(\s(2))}(a_1)\d^1_{2\eta(1)}(a_2) Y_{\mu^k(\g(1))}^{\mu^k(\s(1))}(a_3)  Y_{\mu^k(\g(2))}^{\mu^k(\s(1))}(a_4)\d^2_{1\eta(2)}(a_5)\right) \\
&-\tau\left(a_0Y_{\mu^k(\s(1))}^{\mu^k(\s(2))}(a_1)\d^2_{1\eta(1)}(a_2) Y_{\mu^k(\g(1))}^{\mu^k(\s(1))}(a_3)  Y_{\mu^k(\g(2))}^{\mu^k(\s(1))}(a_4)\d^1_{2\eta(2)}(a_5)\right) \\
&-\tau\left(a_0 Y_{\mu^k(\s(1))}^{\mu^k(\s(2))}(a_1) Y_{\mu^k(\g(1))}^{\mu^k(\s(1))}(a_2)\d^1_{2\eta(1)}(a_3)  \d^2_{1\eta(2)}(a_4)Y_{\mu^k(\g(2))}^{\mu^k(\s(1))}(a_5)\right) \\
&-\tau\left(a_0Y_{\mu^k(\s(1))}^{\mu^k(\s(2))}(a_1) Y_{\mu^k(\g(1))}^{\mu^k(\s(1))}(a_2)\d^2_{1\eta(1)}(a_3)  \d^1_{2\eta(2)}(a_4)Y_{\mu^k(\g(2))}^{\mu^k(\s(1))}(a_5)\right) \\
&+\tau\left(a_0 Y_{\mu^k(\s(1))}^{\mu^k(\s(2))}(a_1) Y_{\mu^k(\g(1))}^{\mu^k(\s(1))}(a_2)\d^1_{2\eta(1)}(a_3)  Y_{\mu^k(\g(2))}^{\mu^k(\s(1))}(a_4)\d^2_{1\eta(2)}(a_5)\right) \\
&+\tau\left(a_0Y_{\mu^k(\s(1))}^{\mu^k(\s(2))}(a_1) Y_{\mu^k(\g(1))}^{\mu^k(\s(1))}(a_2)\d^2_{1\eta(1)}(a_3)  Y_{\mu^k(\g(2))}^{\mu^k(\s(1))}(a_4)\d^1_{2\eta(2)}(a_5)\right) \\
&-\tau\left(a_0 Y_{\mu^k(\s(1))}^{\mu^k(\s(2))}(a_1) Y_{\mu^k(\g(1))}^{\mu^k(\s(1))}(a_2)  Y_{\mu^k(\g(2))}^{\mu^k(\s(1))}(a_3)\d^1_{2\eta(1)}(a_4)\d^2_{1\eta(2)}(a_5)\right) \\
&-\tau\left(a_0Y_{\mu^k(\s(1))}^{\mu^k(\s(2))}(a_1) Y_{\mu^k(\g(1))}^{\mu^k(\s(1))}(a_2)  Y_{\mu^k(\g(2))}^{\mu^k(\s(1))}(a_3)\d^2_{1\eta(1)}(a_4)\d^1_{2\eta(2)}(a_5)\right)\Big\}.
\end{split}
\end{align}
Finally,
\begin{align}
\begin{split}
& \chi_\vp(\mathscr{R}_4)(a_0 \ot a_1 \ot a_2) =\vp(c_2 \ot a_0 \ot a_1 \ot a_2)\\
&=\sum_{\s \in S_2}(-1)^\s  \Big\{\tau(a_0\d^1_{2\s(1)}(a_1)\d^2_{1\s(2)}(a_2)) + \tau(a_0\d^2_{1\s(1)}(a_1)\d^1_{2\s(2)}(a_2))\Big\}.
\end{split}
\end{align}

\begin{remark}{\rm
One knows that the characteristic map  $\chi_\tau:C^\bullet(\Hc, \Cb_\d)\ra C^\bullet(\Ac)$ is injective \cite{ConnMosc98}.  Since $\chi_\vp(\mathscr{TF})$, $\chi_\vp(\mathscr{GV}),$  $\chi_\vp(\mathscr{R}_1),$  $\chi_\vp(\mathscr{R}_2),$  $\chi_\vp(\mathscr{R}_3),$  and $\chi_\vp(\mathscr{R}_4)$,  are all in the range of $\chi_\tau$, as a byproduct  of our study in this paper,  one  calculates cyclic cocycles representing a basis for    $HP^\bullet(\Hc_2, \Cb_\d)$.  }
\end{remark}
\bibliographystyle{amsplain}
\bibliography{Rangipour-Sutlu-References}{}

\end{document}